\documentclass[12pt,draft]{article}
\usepackage[pagewise]{lineno}
\usepackage{amsfonts}
\usepackage{amsmath, amsthm, amsfonts, amssymb, color}
\usepackage{mathrsfs}
\usepackage{color}
\usepackage{stmaryrd}
\newtheorem{thm}{Theorem}[section]

\newtheorem{lem}[thm]{Lemma}

\newtheorem{rem}[thm]{Remark}
\theoremstyle{definition}
\newtheorem{defn}{Definition}[section]
\newcommand{\scr}[1]{\mathscr #1}
\definecolor{wco}{rgb}{0.5,0.2,0.3}

\numberwithin{equation}{section} \theoremstyle{remark}

\newcommand{\ua}{\uparrow}

\title{{\bf Extrinsic derivatives for  SDEs and SPDEs with distribution dependent noise }\footnote{Supported in
 part by NNSFC (12301180), RGC (21301925) and Research Center for Nonlinear Analysis at The Hong Kong Polytechnic University.\ \ }
}
\author{{\bf   Xiaochen Ma,   Panpan Ren}\\ \footnotesize{Department of Mathematics, City University of Hong Kong, Tat Chee Avenue, Hong Kong, China}\\ \footnotesize{   xiaocma2-c@my.cityu.edu.hk },\footnotesize{  panparen@cityu.edu.hk}}
\begin{document}
\allowdisplaybreaks
\def\R{\mathbb R}  \def\ff{\frac} \def\ss{\sqrt} \def\B{\mathbf
B}
\def\N{\mathbb N} \def\kk{\kappa} \def\m{{\bf m}}
\def\ee{\varepsilon}\def\ddd{D^*}
\def\dd{\delta} \def\DD{\Delta} \def\vv{\varepsilon} \def\rr{\rho}
\def\<{\langle} \def\>{\rangle} \def\GG{\Gamma} \def\gg{\gamma}
  \def\nn{\nabla} \def\pp{\partial} \def\E{\mathbb E}
\def\d{\text{\rm{d}}} \def\bb{\beta} \def\aa{\alpha} \def\D{\scr D}
  \def\si{\sigma} \def\ess{\text{\rm{ess}}}
\def\beg{\begin} \def\beq{\begin{equation}}  \def\F{\scr F}
\def\Ric{\text{\rm{Ric}}} \def\Hess{\text{\rm{Hess}}}
\def\e{\text{\rm{e}}} \def\ua{\underline a} \def\OO{\Omega}  \def\oo{\omega}
 \def\tt{\tilde} \def\Ric{\text{\rm{Ric}}}
\def\cut{\text{\rm{cut}}} \def\P{\mathbb P} \def\ifn{I_n(f^{\bigotimes n})}
\def\C{\scr C}   \def\G{\scr G}   \def\aaa{\mathbf{r}}     \def\r{r}
\def\gap{\text{\rm{gap}}} \def\prr{\pi_{{\bf m},\varrho}}  \def\r{\mathbf r}
\def\Z{\mathbb Z} \def\vrr{\varrho} \def\ll{\lambda}
\def\L{\scr L}\def\Tt{\tt} \def\TT{\tt}\def\II{\mathbb I}
\def\i{{\rm in}}\def\Sect{{\rm Sect}}  \def\H{\mathbb H}
\def\M{\scr M}\def\Q{\mathbb Q} \def\texto{\text{o}} \def\LL{\Lambda}
\def\Rank{{\rm Rank}} \def\B{\scr B} \def\i{{\rm i}} \def\HR{\hat{\R}^d}
\def\to{\rightarrow}\def\l{\ell}\def\iint{\int}
\def\EE{\scr E}\def\no{\nonumber}
\def\A{\scr A}\def\V{\mathbb V}\def\osc{{\rm osc}}
\def\BB{\scr B}\def\Ent{{\rm Ent}}\def\3{\triangle}\def\H{\scr H}
\def\U{\scr U}\def\8{\infty}\def\1{\lesssim}\def\HH{\mathrm{H}}
 \def\T{\scr T}
 \def\R{\mathbb R}  \def\ff{\frac} \def\ss{\sqrt} \def\B{\mathbf
B} \def\W{\mathbb W}
\def\N{\mathbb N} \def\kk{\kappa} \def\m{{\bf m}}
\def\ee{\varepsilon}\def\ddd{D^*}
\def\dd{\delta} \def\DD{\Delta} \def\vv{\varepsilon} \def\rr{\rho}
\def\<{\langle} \def\>{\rangle} \def\GG{\Gamma} \def\gg{\gamma}
  \def\nn{\nabla} \def\pp{\partial} \def\E{\mathbb E}
\def\d{\text{\rm{d}}} \def\bb{\beta} \def\aa{\alpha} \def\D{\scr D}
  \def\si{\sigma} \def\ess{\text{\rm{ess}}}
\def\beg{\begin} \def\beq{\begin{equation}}  \def\F{\scr F}
\def\Ric{\text{\rm{Ric}}} \def\Hess{\text{\rm{Hess}}}
\def\e{\text{\rm{e}}} \def\ua{\underline a} \def\OO{\Omega}  \def\oo{\omega}
 \def\tt{\tilde} \def\Ric{\text{\rm{Ric}}}
\def\cut{\text{\rm{cut}}} \def\P{\mathbb P} \def\ifn{I_n(f^{\bigotimes n})}
\def\C{\scr C}      \def\aaa{\mathbf{r}}     \def\r{r}
\def\gap{\text{\rm{gap}}} \def\prr{\pi_{{\bf m},\varrho}}  \def\r{\mathbf r}
\def\Z{\mathbb Z} \def\vrr{\varrho} \def\ll{\lambda}
\def\L{\scr L}\def\Tt{\tt} \def\TT{\tt}\def\II{\mathbb I}
\def\i{{\rm in}}\def\Sect{{\rm Sect}}  \def\H{\mathbb H}
\def\M{\scr M}\def\Q{\mathbb Q} \def\texto{\text{o}} \def\LL{\Lambda}
\def\Rank{{\rm Rank}} \def\B{\scr B} \def\i{{\rm i}} \def\HR{\hat{\R}^d}
\def\to{\rightarrow}\def\l{\ell}\def\iint{\int}
\def\EE{\scr E}\def\Cut{{\rm Cut}}
\def\A{\scr A} \def\Lip{{\rm Lip}}
\def\BB{\scr B}\def\Ent{{\rm Ent}}\def\L{\scr L}
\def\R{\mathbb R}  \def\ff{\frac} \def\ss{\sqrt} \def\B{\mathbf
B}
\def\N{\mathbb N} \def\kk{\kappa} \def\m{{\bf m}}
\def\dd{\delta} \def\DD{\Delta} \def\vv{\varepsilon} \def\rr{\rho}
\def\<{\langle} \def\>{\rangle} \def\GG{\Gamma} \def\gg{\gamma}
  \def\nn{\nabla} \def\pp{\partial} \def\E{\mathbb E}
\def\d{\text{\rm{d}}} \def\bb{\beta} \def\aa{\alpha} \def\D{\scr D}
  \def\si{\sigma} \def\ess{\text{\rm{ess}}}
\def\beg{\begin} \def\beq{\begin{equation}}  \def\F{\scr F}
\def\Ric{\text{\rm{Ric}}} \def\Hess{\text{\rm{Hess}}}
\def\e{\text{\rm{e}}} \def\ua{\underline a} \def\OO{\Omega}  \def\oo{\omega}
 \def\tt{\tilde} \def\Ric{\text{\rm{Ric}}}
\def\cut{\text{\rm{cut}}} \def\P{\mathbb P} \def\ifn{I_n(f^{\bigotimes n})}
\def\C{\scr C}      \def\aaa{\mathbf{r}}     \def\r{r}
\def\gap{\text{\rm{gap}}} \def\prr{\pi_{{\bf m},\varrho}}  \def\r{\mathbf r}
\def\Z{\mathbb Z} \def\vrr{\varrho} \def\ll{\lambda}
\def\L{\scr L}\def\Tt{\tt} \def\TT{\tt}\def\II{\mathbb I}
\def\i{{\rm in}}\def\Sect{{\rm Sect}}  \def\H{\mathbb H}
\def\M{\scr M}\def\Q{\mathbb Q} \def\texto{\text{o}} \def\LL{\Lambda}
\def\Rank{{\rm Rank}} \def\B{\scr B} \def\i{{\rm i}} \def\HR{\hat{\R}^d}
\def\to{\rightarrow}\def\l{\ell}
\def\8{\infty}\def\I{1}\def\U{\scr U}
\maketitle

\begin{abstract} The Bismut formula is a crucial tool characterizing regularities of stochastic systems, and  has been extensively studied for various models.
However it is not yet available for SDEs with distribution dependent noise. In this paper, we first establish a Bismut type formula for the extrinsic derivative of  McKean-Vlasov SDEs driven by distribution dependent noise, then make an extension to a class of  distribution dependent SPDEs.

\end{abstract} \noindent

\tableofcontents 
\section{Introduction}

In 1984, Bismut \cite{BIS} found a derivative formula for the heat  semigroup on a Riemannian manifold by using Malliavin calculus,
which was then re-proved by Elworthy and Li \cite{BIS2} in 1994 using martingale arguments. Since then, this type of formula has been extensively studied and applied to various stochastic systems including SDEs and SPDEs, see for instance \cite{RZ25, T22, HARNACK} and references therein. 

On the other hand, as  crucial stochastic systems describing non-linear Fokker-Planck equations, McKean-Vlasov SDEs (also called distribution dependent SDEs)
have been intensively studied, see for instance   \cite{DH, HRW19}. To characterize the regularity of this type of SDEs with respect to the initial distribution, the Bismut formula has been established for the intrinsic/Lions derivative, see   \cite{BRW21, FH22, HS, HW2209, RZ25, RW19, W23, WRbook}. 
However, in comparison with the Bismut formula for the intrinsic/Lions derivative which naturally links to the Malliavin derivative, it seems harder 
to establish the   Bismut formula for the extrinsic derivative. As far as we know, the only result on this topic was derived by the second named author \cite{PPED},
where the noise is distribution-free. It has already been understood in the study of the intrinsic/Lions derivative that the distribution dependence of noise may cause additional
 difficulty. In this paper, we overcome  this difficulty by modifying the noise decomposition argument  in \cite{HW2209} to establish a Bismut formula for the extrinsic derivative of   SDEs and SPDEs with distribution dependent noise.

Let $\mathscr P$ be the set of all probability measures on $\R^d$. For any $k\in [0,\infty)$, the space of probability measures with finite k-th moment
$$  \mathscr{P}_k:=\big\{\mu\in \mathscr P:\ \mu(|\cdot|^k)<\infty\big\}$$
is a complete metric space  with respect to the weighted total variation
distance
\[
\|\mu-\nu\|_{k,\mathrm{var}}
:=
\sup_{|f|\le 1+|\cdot|^k}
|\mu(f)-\nu(f)|,
\]
see, for instance, \cite{CV}. In particular, $\mathscr{P}=\mathscr{P}_0$.

Following the line of  \cite{HW2209}, we consider the following distribution dependent SDE (DDSDE for short) on $\R^d$:
 \begin{equation}\label{EQ}
\d X_t
=
b_t(X_t,\L_{X_t})\,\d t
+
\lambda\,\d W_t
+
\sigma_t(\L_{X_t})\,\d B_t,
\qquad t\in[0,T],
\end{equation}
where $\L_{X_t}$ denotes the distribution of $X_t$,
  $W_t$ and $B_t$ are independent $d$-dimensional Brownian motions defined on
a complete filtered probability space
$(\Omega,\{\mathscr{F}_t\}_{t\in[0,T]},\mathscr{F},\P)$,
$\lambda\in\R$ is a constant, and for some $k\in [0,\infty)$ 
\[
b:[0,T]\times\R^d\times\mathscr{P}_k\to\R^d,
\qquad
\sigma:[0,T]\times\mathscr{P}_k\to\R^d\otimes\R^d
\]
are measurable. One may reformulate the sum of two noise terms as
$$\lambda\,\d W_t
+
\sigma_t(\L_{X_t})\,\d B_t=\ss{\ll^2+  (\si_t\si_t^*)(\L_{X_t})}\,\d \tt W_t,$$
where $\tt W_t:= \int_0^t (\ll^2 +(\si_s\si_s^*)(\L_{X_s}))^{-\ff 1 2}(\ll \d W_s+ \si_s(\L_{X_s})\d B_s)$ is $d$-dimensional Brownian motion. But the noise decomposition formulation enables us to establish the Bismut formula only using the additive noise $W_t$.

\begin{defn}\label{DEF}
The DDSDE \eqref{EQ} is said to be well-posed for distributions in $\mathscr{P}_k$ if,
for any $\mathscr{F}_0$-measurable initial random variable $X_0$ with
$\L_{X_0}\in\mathscr{P}_k$ (respectively, any initial distribution $\gamma\in\mathscr{P}_k$),
it admits a unique strong (respectively, weak) solution such that
\[
\L_{X_\cdot}\in C([0,T];\mathscr{P}_k).
\]
\end{defn}

When \eqref{EQ} is well-posed, it naturally induces a nonlinear Markov semigroup acting on
probability measures. For $\mu\in\mathscr{P}_k$ and $t\in(0,T]$, set
\[
P_t^*\mu
:=
\L_{X_t^\mu}
\quad\text{with}\quad
\L_{X_0}=\mu,
\]
where $X_t^\mu$ denotes the solution to \eqref{EQ} with initial distribution $\mu$.
It is well known that $P_t^*\mu$ solves the corresponding nonlinear Fokker-Planck equation, so to characterize the regularity of nonlinear Fokker-Planck equations
it is crucial to study the derivative of $P_t^*\mu$ with respect to $\mu\in \mathscr P_k$. 
 Since $P_t^*\mu$ is determined by
 $$P_tf(\mu):= \E\big[f(X^\mu_t)\big],\ \ \ f\in \B_b(\R^d),$$ 
 where $\B_b(\R^d)$ denotes the space of bounded Borel measurable functions on $\R^d$. Studying the regularity of 
\[
\mathscr{P}_k\ni\mu\mapsto P_t^*\mu\] is equivalent to studying that of 
\[
\mathscr{P}_k\ni\mu\mapsto P_t f(\mu),\ \ \ f\in \B_b(\R^d).
\]
In this paper, we calculate the extrinsic derivative of $P_tf(\mu)$ with respect to $\mu\in \mathscr P_k$.

Different notions of differentiability on $\mathscr{P}_k$ correspond to distinct
physical mechanisms in particle systems. The intrinsic derivative describes the
transport or motion of particles and is associated with perturbations induced by
spatial displacements along flows, while the extrinsic derivative reflects the
birth-death or creation-annihilation of particles; see, for instance,
\cite{KLV,PPWSG} and the references therein.

Extrinsic derivative formulas have been established for DDSDEs with distribution independent noise~\cite{PPED}. However, the case where the noise $\sigma$ depends on the distribution $\mathscr  L_{X_t}$
  remains  unexplored. This dependence implies a more complex coupling as the noise term varies with the distribution. In this paper, we address this issue by employing the noise decomposition method. To formulate our results, we first recall the notion of extrinsic differentiability on
$\mathscr{P}_k$.

\begin{defn}\label{DEF1}
Let $f$ be a continuous function on $\mathscr{P}_k$.
\begin{enumerate}
\item[(1)]
We say that $f$ is extrinsically differentiable if, for any
$\mu\in\mathscr{P}_k$, the convex derivative
\[
\tilde{D}_x^E f(\mu)
:=
\lim_{\varepsilon\downarrow0}
\frac{
f\big((1-\varepsilon)\mu+\varepsilon\delta_x\big)-f(\mu)
}{\varepsilon}
\in\R,
\quad x\in\R^d,
\]
exists.
\item[(2)]
We denote by $C^{E,1}(\mathscr{P}_k)$ the class of extrinsically differentiable
functions $f$ such that the map
\[
(x,\mu)\in\R^d\times\mathscr{P}_k
\longmapsto
\tilde{D}^E f(\mu)(x):=\tilde{D}_x^E f(\mu)
\]
is continuous.
\item[(3)]
We write $f\in C^{E,1}_K(\mathscr{P}_k)$ if $f\in C^{E,1}(\mathscr{P}_k)$ and, for any
compact set $\mathscr{K}\subset\mathscr{P}_k$, there exists a constant $c>0$ such that
\[
\sup_{\mu\in\mathscr{K}}
|\tilde{D}^E f(\mu)(x)|
\le
c\bigl(1+|x|^k\bigr),
\quad x\in\R^d.
\]
\end{enumerate}
\end{defn}


The extrinsic derivative in Definition~\ref{DEF1} is based on the perturbation
$\mu\mapsto(1-\varepsilon)\mu+\varepsilon\delta_x$, that is, by injecting an infinitesimal
mass at the spatial location $x$. This construction corresponds to a directional
derivative of $f$ along atomic perturbations of the measure and therefore captures
the first-order sensitivity of $f$ with respect to adding a particle at $x$.
This viewpoint is consistent with the
{linear functional derivative} with respect to the measure argument, which is also
frequently called the {flat derivative}  in the mean-field games
and statistical physics literature;  see, for instance,
\cite{Cardaliaguet18, Martini2023, Monmarche24}.

As shown in \cite[Lemma~3.2]{RW21}, if $f\in C_K^{E,1}(\mathscr{P}_k)$, then the map
$x\mapsto\tilde{D}_x^E f(\mu)$ coincides with the linear functional derivative (flat derivative)
of $f$ at $\mu$. Moreover, for any $\mu,\nu\in\mathscr{P}_k$, the following identity holds:
\begin{equation}\label{eq:flat-FTC}
f(\mu)-f(\nu)
=
\int_0^1
(\mu-\nu)\Bigl(\tilde{D}^E f\bigl(r\mu+(1-r)\nu\bigr)\Bigr)\,\mathrm{d}r,
\end{equation}
which can be viewed as a measure-valued version of the fundamental theorem of calculus along the
linear interpolation between $\mu$ and $\nu$.

The paper is organized as follows. In Section~2, we establish a Bismut type
formula for the extrinsic  derivative of  the DDSDE \eqref{EQ}. In Section~3, we
extend this formula to distribution dependent SPDEs.

Throughout the paper, $C$ (or $c$) denotes a generic nonnegative constant, possibly
differing from line to line.

\section{ SDEs with distribution dependent noise}

To calculate $\tilde{D}^E P_t f$, we use the semigroup associated with the decoupled SDE
\begin{equation}\label{EQ1}
\d X_t^{\mu,x}
=
b_t\big(X_t^{\mu,x}, P_t^{*}\mu\big)\,\d t
+
\lambda\,\d W_t
+
\sigma_t\big(P_t^{*}\mu\big)\,\d B_t,
\qquad X_0^{\mu,x}=x,\ t\in[0,T],
\end{equation}
where $x\in\mathbb{R}^d$ and $\mu\in\mathscr{P}_k$.
Let $P_t^{\mu}$ be the associated semigroup, i.e.
\[
P_t^{\mu}f(x):= \mathbb{E}\big[f(X_t^{\mu,x})\big],
\quad x\in\mathbb{R}^d,\ t\in[0,T].
\]
For a probability measure $\nu$ on $\mathbb{R}^d$, define
\[
P_t^{\mu}f(\nu):= \int_{\mathbb{R}^d} P_t^{\mu}f(x)\,\nu(\d x),
\qquad
(P_t^{\mu})^{*}\nu := \int_{\mathbb{R}^d} \mathscr{L}_{X_t^{\mu,x}}\,\nu(\d x).
\]

We impose the following assumptions.

\begin{enumerate}
\item[\textbf{(H)}]
Let $k\in[1,\infty)$. The following conditions hold for some constant $K\in (0,\infty)$.

\medskip

\noindent\textbf{(H1)}
For all $t\in[0,T]$, $x,y\in\mathbb{R}^d$ and $\mu,\nu\in\mathscr{P}_k(\mathbb{R}^d)$,
\begin{equation}\label{hh2}
\begin{aligned}
|b_t(x,\mu)-b_t(y,\nu)| &\le K\big(|x-y|+\|\mu-\nu\|_{k,\mathrm{var}}\big),\\
\|\sigma_t(\mu)-\sigma_t(\nu)\| &\le K\,\mathbb{W}_k(\mu,\nu),\\
\|\sigma_t(\delta_0)\| +|b_t(0,\delta_0)|&\le K.\\
\end{aligned}
\end{equation}

\medskip

\noindent\textbf{(H2)}
$b_t(x,\mu)$ and $\sigma_t(\mu)$ belong to $C^{E,1}(\mathscr{P}_k)$.
Moreover, for all $t\in[0,T]$, $y\in\mathbb{R}^d$, $\mu\in\mathscr{P}_k$,
\begin{equation}\label{hh3}
\inf_{c\in\mathbb{R}^d}
\Big(
|\tilde{D}^E b_t(x,\mu)(y)-c|
+
\|\tilde{D}^E \sigma_t(\mu)(y)-c\|
\Big)
\le
K\,(1+|y|^k).
\end{equation}

\medskip

\noindent\textbf{(H3)}
There exists an increasing function $\alpha:(0,\infty)\to(0,\infty)$ with $\alpha(\varepsilon)\to0$ as $\varepsilon\to0$ such that for all $t\in[0,T]$, $x,y\in\mathbb{R}^d$ and $\mu,\nu\in\mathscr{P}_k$,
\begin{equation}\label{hh4}
\begin{aligned}
&\big|
\tilde{D}^E b_t(x,\mu)(y)
-
\tilde{D}^E b_t(x,\nu)(y)
\big|
+
\big\|
\tilde{D}^E \sigma_t(\mu)(y)
-
\tilde{D}^E \sigma_t(\nu)(y)
\big\|
\\
&\qquad\le
\alpha\big(\|\mu-\nu\|_{k,\mathrm{var}}\big)
\big(
1 + |y|^k + \mu(|\cdot|^k) + \nu(|\cdot|^k)
\big).
\end{aligned}
\end{equation}

\end{enumerate}

\begin{rem}
As illustrated by \emph{\cite[Example~1.1]{DH}}, the well-posedness of \eqref{EQ1} may fail if   $\mathbb{W}_k$ is replaced by the weighted variation distance $\|\cdot\|_{k,\mathrm{var}}$ in condition \eqref{hh2}.

On the other hand, by \emph{\cite[Lemma~2.1]{PPED}},    \eqref{hh3} implies 
\begin{equation*}
\begin{aligned}
|b_t(x,\mu)-b_t(x,\nu)|
&\le C\,\|\mu-\nu\|_{k,\mathrm{var}}, \\[0.5ex]
\|\sigma_t(\mu)-\sigma_t(\nu)\|
&\le C\,\|\mu-\nu\|_{k,\mathrm{var}},
\end{aligned}
\end{equation*}
for all $t\in[0,T]$, $x\in\mathbb{R}^d$ and $\mu,\nu\in\mathscr{P}_k$.  Since    $\mathbb{W}_k$ and $\|\cdot\|_{k,\mathrm{var}}$ are not comparable for $k>1$,   conditions      \eqref{hh2} and  \eqref{hh3} are not comparable as well. 

\end{rem}

Let $\mathscr{D}_k$ denote the class of measurable functions $f$ on $\R^d$
such that $|f(x)|\le c(1+|x|^k)$ for some constant $c>0$. Let $\mathscr{M}_T$ be the set of all measurable maps
\[
\eta = (\eta_{s,t})_{0 \leq s \leq t \leq T}: \{(s,t): 0 \leq s \leq t \leq T\} \times \Omega \to \mathbb{R}^d
\]
such that each $\eta_{s,t}$ is $\mathscr{F}_t$-measurable and
\[
\|\eta\|_{\mathscr{M}_T} := \sup_{t \in [0,T]} \left( \int_0^t \mathbb{E} |\eta_{s,t}|^2 ds \right)^{\frac{1}{2}}<\infty .
\]
The main result of this paper is the following.
\begin{thm}\label{MR} 
Assume {\bf (H)}. Then \eqref{EQ} is well-posed for any
initial distribution $\mu \in \mathscr{P}_k$, and we denote by $X_t^{\mu}$ the
(unique) solution to \eqref{EQ} with initial distribution $\mu$. Moreover, the
following assertions hold.

\begin{enumerate}
	\item[(1)] For any fixed $t>0$ and any $\mu,\nu\in \mathscr{P}_k$, $s \leq t$, there exists unique $\eta \in \mathscr{M}_T$ such that
	\begin{equation*}
	\eta^{\mu,\nu}_{s,t}
	=
	J_s
	+
	\Big\langle \nabla b(\cdot, P_{s}^{*}\mu)(X_{s}^{\mu}), H_s-\frac{s}{t}H_t \Big\rangle
	+
	\frac{1}{t}H_t,\quad 0 \leq s \leq t \leq T,
	\end{equation*}
	where
	\begin{equation*}
	\begin{aligned}
	J_s :=&\, P_{s}^{\mu}\Big(\tilde{D}^E b_{s}(X_{s}^{\mu},P^{*}_{s}\mu)(\mu) - \tilde{D}^E b_{s}(X_{s}^{\mu},P^{*}_{s}\nu)(\mu)\Big)\\
	& - \mathbb{E}\Big[\tilde{D}^E b_{s}(z, P_s^*\mu)(X_{s}^{\mu})
	\int_{0}^{s} \Big\langle \frac{1}{\lambda}\eta^{\mu,\nu}_{r,s},\d W_r \Big\rangle \Big]_{z=X_{s}},\\
	H_t :=&\, \int_{0}^{t} P_{r}^{\mu} \Big(\tilde{D}^E \sigma_{r}(P^{*}_{r}\mu)(\mu) - \tilde{D}^E \sigma_{r}(P^{*}_{r}\mu)(\nu)\Big)\,\d B_r \\
	& - \int_{0}^{t} \mathbb{E} \Big[\tilde{D}^E \sigma_{r}(P^{*}_r\mu)(X_{r}^{\mu}) \int_{0}^{r} \Big\langle \frac{1}{\lambda}\eta^{\mu,\nu}_{u,r},\d W_u \Big\rangle \Big]\,\d B_r.
	\end{aligned}
	\end{equation*}
	Consequently, there exists a constant $c>0$ such that
\begin{equation}\label{US}
\mathbb{E}\left|\eta_{s,t}^{\mu,\nu}\right|^2 
\leq c \left( 1 + \bigl(\mu + \nu\bigr)\bigl(| \cdot |^k\bigr) \right)^2
\exp\left\{\frac{c}{\lambda^2}\bigl( 1 + \mu\bigl(| \cdot |^k\bigr) \bigr)^2\right\},
\quad  t\in [0,T].
\end{equation}

	\item[(2)]  For any $f\in\mathscr{D}_k$ and  $t\in (0,T]$, $P_t f$ is extrinsically differentiable and 
	\begin{equation}\label{DENU}
	\tilde{D}^E_{\nu} P_t f(\mu)
	=
	\int_{\R^d} P_t^{\mu} f(x)\, (\nu-\mu)(\d x)
	+
	\mathbb{E}\Big[f(X_t^{\mu}) \int_0^t \Big\langle \frac{1}{\lambda}\eta_{s,t}^{\mu,\nu},\d W_s \Big\rangle \Big].
	\end{equation}
    As a consequence, we can find a constant $c > 0$ such that the following estimate holds:
	\begin{equation}\label{SED}
	\begin{aligned}
	\sup_{|f(x)|\le 1+|x|^k} |\tilde{D}^E_{\nu} P_t f(\mu)|
	\le &\, c \bigl(1 + \mu(|\cdot|^k)\bigr) \bigl(1 + (\mu + \nu)(|\cdot|^k)\bigr) 
	\exp\Big\{\frac{c}{\lambda^2 t}\bigl(1 + \mu(|\cdot|^k)\bigr)^2\Big\} t^{\frac{1}{2}} \\
	& + c \bigl(1 + (\mu + \nu)(|\cdot|^k)\bigr), 
	\quad t\in [0,T].
	\end{aligned}
	\end{equation}
\end{enumerate}
\end{thm}
{{
\begin{rem}

To the best of our knowledge, this is the first time that a Bismut-type extrinsic
derivative formula is established for DDSDEs with distribution dependent diffusion
coefficients. The key novelty lies in the construction of the auxiliary process
$\eta^{\mu,\nu}_{s,t}$, which satisfies a closed stochastic integral equation and
admits the uniform estimate \eqref{US}. This representation enables us to derive
the explicit extrinsic derivative formula \eqref{DENU} together with the quantitative
estimate \eqref{SED}.
\end{rem}}}

\subsection{Decoupled SDE and stability estimates}

For any $\mu_{\cdot}\in C\left([0,T];\mathscr{P}_k\right)$, consider the SDE:
\begin{equation}\label{FrozenSDE}
\mathrm d X_t
= b_t(X_t,\mu_t)\,\mathrm dt
+ \lambda\,\mathrm d W_t
+ \sigma_t(\mu_t)\,\mathrm d B_t,\quad \mu_{\cdot}, \mu_{\cdot}\in\mathscr{C}_T^{\gamma},
t\in[0,T].
\end{equation}
Obviously, \eqref{FrozenSDE} is well-posed due to  {(\textbf{H1})}, and we can derive the following estimates.

\begin{lem}\label{lem:lip-estimate}
Assume \emph{(\textbf{H1})} holds and let $p>k\ge 1$. Then there exists a constant $C>0$
such that for any $\mathscr F_0$-measurable $\mathbb R^d$-valued random variable $X_0$
with $\mathbb E|X_0|^k<\infty$, and any $\mu^i,\nu^i\in C\!\left([0,T];\mathscr P_k\right)$,
$i=1,2$, the following estimates hold:
\[ \begin{aligned} \mathbb{W}_k\!\Big(\mathscr L_{X_t^{1}}, \mathscr L_{X_t^{2}}\Big)^k &\le C \left( \begin{aligned} &\int_0^t \|\mu_s^1-\mu_s^2\|_{k,\mathrm{var}}^k \,\mathrm{d}s + \int_0^t \mathbb{W}_k(\mu_s^1,\mu_s^2)^k \,\mathrm{d}s \end{aligned} \right), \quad t\in[0,T], \\ \|\mathscr L_{X_t^{1}} - \mathscr L_{X_t^{2}}\|_{k,\mathrm{var}} &\le C \left(1 + \mathbb{E}|X_0|^p + \int_0^t \mu_s^1(|\cdot|^k)^{\frac{p}{k}}\,\mathrm{d}s \right)^{\frac{k}{p}} \\ &\qquad \times \left( \begin{aligned} &\int_0^t \|\mu_s^1-\mu_s^2\|_{k,\mathrm{var}}^2 \,\mathrm{d}s + \frac{1}{t}\int_0^t \mathbb{W}_k(\mu_s^1,\mu_s^2)^2 \,\mathrm{d}s \end{aligned} \right)^{\frac12}, \quad t\in[0,T], \end{aligned} 
\]
Here $X^i=(X_t^i)_{t\in[0,T]}$ is the solution to \eqref{FrozenSDE} with initial value $X_0$
corresponding to $\mu^i$, for $i=1,2$.

\end{lem}
\begin{proof}

For $i=1,2$, let $X^i=(X_t^i)_{t\in[0,T]}$ be solutions to
\begin{align*}
\mathrm d X_t^{1}
&=
b_t(X_t^{1},\mu_t^1)\,\mathrm dt
+\lambda\,\mathrm dW_t
+\sigma_t(\mu_t^1)\,\mathrm dB_t,
\quad X_0^{1}=X_0,\\
\mathrm d X_t^{2}
&=
b_t(X_t^{2},\mu_t^2)\,\mathrm dt
+\lambda\,\mathrm dW_t
+\sigma_t(\mu_t^2)\,\mathrm dB_t,
\quad X_0^{2}=X_0 .
\end{align*}
Subtracting the above two equations, 
applying It\^o's formula to 
$|X_t^{1}-X_t^{2}|^k$ 
and taking expectation, we obtain
\begin{align*}
\mathbb E|X_t^{1}-X_t^{2}|^k
&\le
k\int_0^t
\mathbb E\!\left[
|X_t^{1}-X_t^{2}|^{k-2}
\Big\langle
X_t^{1}-X_t^{2},
b_s(X_t^{1},\mu_s^1)-b_s(X_t^{2},\mu_s^2)
\Big\rangle
\right]\mathrm ds \notag\\
&\quad
+\frac{k(k-1)}{2}
\int_0^t
\mathbb E\!\left[
|X_t^{1}-X_t^{2}|^{k-2}
\|\sigma_s(\mu_s^1)-\sigma_s(\mu_s^2)\|^2
\right]\mathrm ds .
\end{align*}
By assumption \textbf{(H1)}, and applying Young's inequality, we obtain
\[
\mathbb E|X_t^{1}-X_t^{2}|^k
\le
c_1
\int_0^t
\mathbb E|X_t^{1}-X_t^{2}|^k\,\mathrm ds
+
c_1
\int_0^t
\Big(
\|\mu_s^1-\mu_s^2\|_{k,\mathrm{var}}^k
+
\mathbb W_k(\mu_s^1,\mu_s^2)^k
\Big)\mathrm ds,
\]
for a constant $c_1>0$. The linear growth condition in \eqref{hh2} ensures $\mathbb{E}|X_s^{i}|^k<\infty$, hence $\mathbb E|X_t^{1}-X_t^{2}|^k$ is finite.  
Applying Gronwall's inequality yields
\[
\mathbb E|X_t^{1}-X_t^{2}|^k
\le
\mathrm e^{c_1t}
\Big(
\int_0^t
\|\mu_s^1-\mu_s^2\|_{k,\mathrm{var}}^k\,\mathrm ds
+
\int_0^t
\mathbb W_k(\mu_s^1,\mu_s^2)^k\,\mathrm ds
\Big).
\]
Finally, since
\[
\mathbb W_k\!\left(
\mathscr L_{X_t^{1}},
\mathscr L_{X_t^{2}}
\right)^k
\le
\mathbb E|X_t^{1}-X_t^{2}|^k,
\]
the proof for the ${L}^k$-Wasserstein distance estimate is complete.

For the weighted total variation estimate, we use the conditional probability and conditional expectation given $\mathscr{F}_0$ and Brownian motion $B$:
\[
\mathbb{P}^{B} := \mathbb{P}(\,\cdot\, \mid B, \mathscr{F}_0), \quad \mathbb{E}^{B} := \mathbb{E}(\,\cdot\, \mid B, \mathscr{F}_0).
\]
For any $t \in [0,T]$, $\mu \in \mathscr{P}_k$, and $f \in \mathscr{B}_b(\mathbb{R}^d)$, let
\[
P_t^{B}f(X_0^\mu) := \mathbb{E}^{B}[f(X_t^{\mu})] = \mathbb{E}\big[f(X_t^{\mu}) \big| B, \mathscr{F}_0\big],
\]
so that
\begin{equation}
P_t f(\mu) = \mathbb{E}\big[P_t^{B}f(X_0^\mu)\big], \quad t \in [0,T], \mu \in \mathscr{P}_k, f \in \mathscr{B}_b(\mathbb{R}^d).
\end{equation}
Next, for any $\mu_{\cdot}\in C([0,T];\mathscr{P}_k)$, define
\begin{equation}\label{XI}
\xi ^{\mu}(t):=\int_0^{t}{\sigma_s(\mu_s)} \mathrm d B_s,\ \ t\in [0,T].
\end{equation}
For a fixed $t_0\in(0,T]$, consider the SDE:
\begin{equation}\label{YEQ}
\begin{aligned}
\mathrm d {Y_t} &= b_t(X^{1}_t,\mu^1_t)\mathrm d t + \frac{\xi^{\mu^1}(t_0)-\xi^{{\mu^2}}(t_0)}{t_0} \mathrm d t\\
&\quad+\lambda\mathrm d W_t+\sigma_t(\mu^2_t)\mathrm d B_t,\ \  {Y_0}={X}_0, \ t\in[0,t_0].
\end{aligned}
\end{equation}
Since
\begin{equation*}
\begin{aligned}
\mathrm d {X^{1}_t} &=b_t(X^{1}_t,\mu^1_t)\mathrm d t +\lambda\mathrm d W_t+\sigma_t(\mu^1_t)\mathrm d B_t,\ \ t\in[0,t_0].
\end{aligned}
\end{equation*}
This together with \eqref{XI} and \eqref{YEQ} yields
\begin{equation}\label{YXEQ}
\begin{aligned}
{Y_t}-X^{1}_t&=\frac{t}{t_0}\left[\xi^{\mu^1}(t_0)-\xi^{{\mu^2}}(t_0)\right]+\xi^{\mu^2}(t)-\xi^{\mu^1}(t),\ \ t\in[0,t_0].
\end{aligned}
\end{equation}
Let
\begin{equation*}
\begin{aligned}
\phi(t):=&b_t(X^{1}_t,\mu^1_t)-b_t({Y_t},\mu^2_t)+\frac{1}{ t_0}\left[\xi^{\mu^1}(t_0)-\xi^{\mu^2}(t_0)\right],\ \ t\in[0,t_0].
\end{aligned}
\end{equation*}
According to {\bf(H1)}, 
\begin{equation}\label{PH2}
\begin{aligned}
|\phi(t)|^2&\leq \frac{1}{t^2_0}\left|\xi^{\mu^1}(t_0)-\xi^{{\mu^2}}(t_0)\right|^2 +2 K^2\|\mu^1_t-\mu^2_t\|_{k,\mathrm{var}}^2
 +2 K^2 |{Y_t}-X^{1}_t|^2\\
 &\leq\frac{1+2K^2t^2}{t^2_0}\left|\xi^{\mu^1}(t_0)-\xi^{\mu^2}(t_0)\right|^2+4 K^2\left|\xi^{\mu^1}(t)-\xi^{\mu^2}(t)\right|^2\\
 &\quad+4 K^2\|\mu^1_t-\mu^2_t\|_{k,\mathrm{var}}^2 ,\ \ t\in[0,t_0].
\end{aligned}
\end{equation}
It\^o isometry and {\bf(H1)} yield
\begin{equation}\label{XX}
\begin{aligned}
&\ \ \ \ \mathbb{E} \left|\xi^{\mu^1}(t)-\xi^{\mu^2}(t)\right|^2 \leq K^2\int_0^t \mathbb{W}_k(\mu_s^1,\mu_s^2)^2 \mathrm ds.
\end{aligned}
\end{equation}
Let $\mathrm d\mathbb Q^{B}:= R\mathrm d\mathbb P^{B}$, where
\begin{equation*}
R:=\exp\left(\int_{0}^{t_0}\langle\frac{1}{\lambda}\phi(s),\mathrm d W_s \rangle-\frac{1}{2}\int_{0}^{t_0}|\frac{1}{\lambda}\phi(s)|^2\mathrm d s\right).
\end{equation*}
By Girsanov's theorem, under the weighted conditional probability $\mathbb Q^{B}$,
\[
\tilde{W}_t:=W_t-\int_{0}^{t}\frac{1}{\lambda}\phi(s)\mathrm d s,\ \ t\in[0,t_0]
\]
is a $d$-dimensional Brownian motion. By \eqref{YEQ}, $\hat{Y}_t := Y_t - \xi^{\mu^2}(t), t \in [0, t_0]$ satisfies the SDE:
\[
\mathrm{d}\hat{Y}_t = b_t( \hat{Y}_t + \xi^{\mu^2}(t), \mu_t^2)\mathrm{d}t + \lambda \mathrm{d}\tilde{W}_t,  \quad t \in [0, t_0].
\]
Let
\[
\hat{X}_t := X_t^{2} - \xi^{\mu^2}(t), \quad t \in [0, t_0].
\]
Then we have
\[
\mathrm{d}\hat{X}_t = b_t( \hat{X}_t + \xi^{\mu^2}(t), \mu_t^2)\mathrm{d}t + \lambda \mathrm{d}W_t, \quad \hat{X}_0 = {X}_0,  t \in [0, t_0].
\]
From weak uniqueness, we know that
\[
\mathscr{L}_{\hat{Y}_{t_0} \mid \mathbb{Q}^{B}} = \mathscr{L}_{\hat{X}_{t_0} \mid \mathbb{P}^{B}}.
\]
Since $\xi^{\mu^2}(t_0)$ is deterministic given $B$, we can obtain that
\[
\mathscr{L}_{{Y}_{t_0} \mid \mathbb{Q}^{B}} = \mathscr{L}_{\hat{Y}_{t_0} + \xi^{\mu^2}(t_0) \mid \mathbb{Q}^{B}} = \mathscr{L}_{\hat{X}_{t_0} + \xi^{\mu^2}(t_0) \mid \mathbb{P}^{B}} = \mathscr{L}_{X^{2}_{t_0} \mid \mathbb{P}^{B}} .
\]
Moreover, from \eqref{YXEQ}, we know that $X^{ 1}_{t_0} = Y_{t_0}$, and therefore,
\[
\mathbb{E}^{B}\left[f(X^{2}_{t_0})\right] = \mathbb{E}_{\mathbb{Q}^{B}}\left[f(Y_{t_0})\right] = \mathbb{E}^{B}\left[Rf(X^{1}_{t_0})\right], \quad f \in \mathscr{B}_b\left(\mathbb{R}^d\right).
\]
Weak uniqueness and H\"{o}lder's inequality yield
\begin{equation}\label{Kvar}
\begin{aligned}
&\ \ \ \ \|\mathscr{L}_{X_t^{1}}-\mathscr{L}_{X_t^{2}}\|_{k,\mathrm{var}}\\
&= \sup_{|f|\leq 1 + |\cdot|^k}\mathbb{E}\left[\mathbb{E}^{B}[f(X^{1}_{t_0})(R-1)]\right]\\
&\leq \mathbb{E}\left[\left(\mathbb{E}^{B}(1+|X^1_{t_0}|^p)\right)^{\frac{k}{p}}\left(\mathbb{E}^{B}(|R-1|^{\frac{p}{p-k}})\right)^{\frac{p-k}{p}}\right].
\end{aligned}
\end{equation}
By It\^o's formula for \( (1 + |X^1_{t}|^k)^{\frac{p}{k}} \), we can find a constant $c_2>0$ such that
	 \beq\label{M1} \E[|X_{t}^{1}|^p]\leq c_2\left(1+\mathbb{E}|X_0|^p+\int^t_0\mu^1_s(|\cdot|^k)^{\frac{p}{k}}\d s\right),\ \ t\in[0,T]. \end{equation}
Then there exists a constant $c_3(p)>0$ such that
\begin{equation}\label{eq:moment}
\begin{split}
\left(\mathbb{E}^B\big(1+|X_{t}^{1}|^p\big)\right)^{\frac{k}{p}}
&\le
c_3(p)\Bigg(
1+\mathbb{E}|X_0|^p
+\int_0^t \mu^1_s(|\cdot|^k)^{\frac{p}{k}}\,\mathrm{d}s\Bigg)^{\frac{k}{p}} .
\end{split}
\end{equation}
Using the inequality $|\mathrm e^r-1|\le(\mathrm e^r+1)|r|$ together with H\"{o}lder's inequality, we can find a constant $c_4>0$ such that
\begin{align*}
\mathbb{E}^B|R-1|^{\frac{p}{p-k}}
\le
c_4\,\mathbb{E}^B\!\left[
\left|
\int_0^{t_0}\Big\langle \tfrac{1}{\lambda}\phi(s),\mathrm dW_s\Big\rangle
\right|^{\frac{p}{p-k}}
+
\left(
\int_0^{t_0}\Big|\tfrac{1}{\lambda}\phi(s)\Big|^2\mathrm ds
\right)^{\frac{p}{2(p-k)}}
\right]
\end{align*}
 Applying the BDG's inequality yields
\begin{equation}\label{Rterm}
\left(\mathbb{E}^B(|R-1|^{\frac{p}{p-k}})\right)^{\frac{p-k}{p}}
\le
\frac{c_5}{\lambda}
\left(
\mathbb{E}^B\int_0^{t_0}|\phi(s)|^2\mathrm ds
\right)^{\frac{1}{2}}
\end{equation}
for some constant $c_5>0$. By \eqref{PH2} and \eqref{XX}, we have
\begin{align}\label{eq:phi}
\mathbb{E}\int_0^{t_0}|\phi(s)|^2\mathrm ds
\le
\left( \frac{K^2}{t_0} + \frac{8K^4 t_0}{3} \right) \int_0^{t_0} \mathbb{W}_k(\mu_s^1, \mu_s^2)^2 \mathrm{d}s
+ 2K^2 \int_0^{t_0} \mathbb{E} \|\mu_s^1 - \mu_s^2\|_{k,\mathrm{var}}^2 \mathrm{d}s.
\end{align}
Substituting \eqref{eq:moment}, \eqref{Rterm} and \eqref{eq:phi} into \eqref{Kvar} completes the proof.
\end{proof}
\subsection{Well-posedness of \eqref{EQ}}
\begin{lem}\label{WP1} Assume {\bf (H)}.Then the DDSDE \eqref{EQ} is well-posed for distributions in $\mathscr{P}_k$, and there
	exists a constant $C>0$  such that for any $\mu\in\mathscr{P}_k$,
	 \begin{equation*}
	 \E[|X_t^{\mu}|^p|\F_0]\leq C(1+|X_0^{\mu}|^p+\mu(|\cdot|^k)^{\frac{p}{k}}). 
	 \end{equation*}
\end{lem}

\begin{proof}
Fix an initial distribution $\gamma\in\mathscr{P}_k$ and let
\[
\mathscr{C}_T^\gamma
:=\big\{\mu_\cdot\in C([0,T];\mathscr{P}_k):\ \mu_0=\gamma\big\}.
\]
For any $\mu_\cdot\in\mathscr{C}_T^\gamma$, consider the  SDE \eqref{FrozenSDE}. The map $\Phi:\mathscr{C}_T^\gamma\to\mathscr{C}_T^\gamma$ is defined by 
  $$\Phi_t(\mu):= \L_{X^{\mu_{\cdot}}_t}, \ \ t\in[0,T]. $$
To prove that $\Phi$ has a unique fixed point in $\mathscr{C}_T^\gamma$, we need to restrict the map to the following bounded subspaces of $\mathscr{C}_T^\gamma$:
 $$\C^{\gamma}_N:=\left\{\mu_{\cdot}\in \mathscr{C}_T^\gamma:\sup_{t\in [0,T]}\e^{-Nt}\mu_t(|\cdot|^k)\leq N (1+\gamma(|\cdot|^k))\right\}.$$
 For each $\mu_\cdot\in\mathscr{C}_T^\gamma$, \eqref{FrozenSDE} admits a unique strong solution satisfying
   \begin{equation*}
   \E[|X_t^{{\mu}_{\cdot}}|^k]\leq c_2\left(1+\gamma(|\cdot|)^k+ \int^t_0\mu_s(|\cdot|^k)\d s\right),\ \ t\in[0,T]
\end{equation*}
due to \eqref{M1}. Then there exists a constant $N_0 >0$, such that for any $N > N_0$, the set $\C^{\gamma}_N$ is invariant under the map $\Phi$.\\
For any $\mu^i \in \C^{\gamma}_N, i=1,2$ , we apply  Lemma~\ref{lem:lip-estimate}, then there exists a constant $c_6(N)>0$ depending on $N$ such  that
\begin{align}\label{Step2-TV2}
\nonumber&\quad\|\Phi_{t}(\mu^1)-\Phi_{t}(\mu^2)\|_{k,\mathrm{var}}+\mathbb{W}_k(\Phi_{t}(\mu^1),\Phi_{t}(\mu^2))\\
&\leq
c_6(N)\Bigg(
\int_0^t \|\mu_s^1-\mu_s^2\|_{k,\mathrm{var}}^2 \,\mathrm{d}s
+ \frac{1}{t}\int_0^t \mathbb{W}_k(\mu_s^1,\mu_s^2)^2 \,\mathrm{d}s
\Bigg)^{{\frac{1}{2}}}\\
\nonumber&+c_6(N) \Bigg(
\int_0^t \|\mu_s^1-\mu_s^2\|_{k,\mathrm{var}}^k \,\mathrm{d}s
+ \int_0^t \mathbb{W}_k(\mu_s^1,\mu_s^2)^k \,\mathrm{d}s
\Bigg)^{{\frac{1}{k}}}
\end{align}
holds for any $t\in(0,T]$. For $\theta>0$, define a weighted distance on $\mathscr{C}_T^\gamma$ by
\[
d_\theta(\mu,\nu)
:=
\sup_{t\in[0,T]}
\mathrm{e}^{-\theta t}[
\|\mu_t-\nu_t\|_{k,\mathrm{var}}+\mathbb{W}_k(\mu_t,\nu_t)].
\]
It is straightforward to check that $(\mathscr{C}_T^\gamma,d_\theta)$ is a complete metric space. Multiplying both sides of \eqref{Step2-TV2} by $\mathrm{e}^{-\theta t}$, we obtain
\begin{align*}
&\quad \mathrm{e}^{-\theta t} \Big(
\|\Phi_{t}(\mu^1)-\Phi_{t}(\mu^2)\|_{k,\mathrm{var}}
+ \mathbb{W}_k(\Phi_{t}(\mu^1),\Phi_{t}(\mu^2))
\Big) \\
&\le c_6(N)
\Bigg[
\int_0^{t} \mathrm{e}^{-2\theta t} \|\mu_s^1 - \mu_s^2\|_{k,\mathrm{var}}^2 ds
+ \frac{1}{t} \int_0^{t} \mathrm{e}^{-2\theta t}\mathbb{W}_k(\mu_s^1,\mu_s^2)^2 \d s
\Bigg]^{1/2} \\
&\quad + c_6(N) \Bigg[
\int_0^{t} \mathrm{e}^{-k\theta t} \|\mu_s^1 - \mu_s^2\|_{k,\mathrm{var}}^k \d s
+ \int_0^{t} \mathrm{e}^{-k\theta t} \mathbb{W}_k(\mu_s^1, \mu_s^2)^k \d s
\Bigg]^{1/k}\\
& \leq c_7(N) \left[ \left(
    \frac{1 - \mathrm{e}^{-2\theta t}}{2t\theta}
   \right)^{\frac{1}{2}}+\left(\frac{1 - \mathrm{e}^{-k\theta t}}{k\theta}\right)^{\frac{1}{k}} \right]d_\theta(\mu^1, \mu^2)
\end{align*}
For some constant $c_7(N)>0$. Taking the supremum over $t\in(0,T]$ and we can choose $\theta>0$ sufficiently large such that the map $\Phi$ becomes a strict contraction. Hence the SDE \eqref{EQ} is well-posed for distributions in $\mathscr{P}_k$.
\end{proof}
We derive
 \begin{equation*}
 \E[|X_t^{\mu}|^p]\leq c(1+|X_0^{\mu
 }|^p+\int^t_0(\E[|X^{\mu}_s|^k])^{\frac{p}{k}}\d s),\ \ t\in[0,T]
\end{equation*}
due to \eqref{M1}. It follows that
 \beq\label{M5} 
\E[\sup_{t\in[0,T]}|X_t^{\mu}|^p|\F_0]\leq c_p(1+|X_0^{\mu
 }|^p),\ \ p\in [1,\infty ), \mu\in\mathscr{P}_k
\end{equation}
for a constant $c_p$. Therefore,
\beq\label{M6} 
\int_{\R^d}|\cdot|^p \d (P^{\mu}_t)^{*}\nu\leq c_p(1+\nu(|\cdot|)^p),\ \ p\in [1,\infty ), \mu, \nu \in\mathscr{P}_k.
\end{equation}
Denote $\pi_{\varepsilon}=(1-\varepsilon)\mu+\varepsilon\nu, \varepsilon\in(0,1)$, then there exists a constant $c_k$ such that
\beq\label{MN} 
	\|{P_s^{*}\pi_{\varepsilon}}-{P_s^{*}\mu}\|_{k,var}\leq c_k(\mu+\nu)(|\cdot|)^k,\ \ \mu, \nu \in\mathscr{P}_k. 
\end{equation}

{{\begin{rem}
We establish the well-posedness result for the considered DDSDE under Condition {\bf{(H)}}, and this result has not been covered by any existing literature.  

\end{rem}}}
{{\begin{rem}
In this subsection we prove the well-posedness of the DDSDE~\eqref{EQ} in $\mathscr P_k$ by a fixed point
argument. Using the $k$-moment estimate induced by \eqref{M1}, we introduce the
invariant bounded subset $\mathscr C_N^\gamma\subset\mathscr C_T^\gamma$.
To obtain uniqueness, we equip $\mathscr C_T^\gamma$ with the weighted distance
\[
d_\theta(\mu,\nu)
:=
\sup_{t\in[0,T]}
\mathrm{e}^{-\theta t}\Big(
\|\mu_t-\nu_t\|_{k,\mathrm{var}}+\mathbb W_k(\mu_t,\nu_t)
\Big),
\]
under which $(\mathscr C_T^\gamma,d_\theta)$ is complete. The stability estimate \eqref{Step2-TV2}
implies that, for $\theta$ sufficiently large, $\Phi$ is a strict contraction on $\mathscr C_N^\gamma$ in the
metric $d_\theta$. Therefore, $\Phi$ admits a unique fixed point, yielding existence and uniqueness of the
solution.\end{rem}}}

\subsection{Lemmas}
\begin{lem}\label{WP} Assume {\bf (H)}. Then there exists a constant $C(\mu,\nu)>0$ increasing in $(\mu+\nu)(|\cdot|^k)$ such that \\
	\begin{equation*}
	\|(P_t^{\pi_\varepsilon})^{*}\gamma - (P_t^{\mu})^{*}\gamma\|_{k,var}\leq C(\mu,\nu)(1+\gamma(|\cdot|)^k)\varepsilon,\ \ \gamma\in\mathscr{P}_k. 
\end{equation*}
\end{lem}

\begin{proof}
	Let $X_0$ be $\F_0$ -measurable with $\L_{X_0}=\gamma.$ Consider the SDE
	\begin{equation*}
	\d X_t=b_t(X_t,P_t^{*}\mu)\d t+\lambda\d W_t+\sigma_t(P_t^{*}\mu)\d B_t,\ \  t\in[0,T],
\end{equation*}
and the SDE
\begin{equation*}\begin{aligned} 
\d Y_t&=b_t(X_t,P_t^{*}\mu)\d t+\frac{1}{t_0}[\xi^{\mu}(t_0)-\xi ^{\pi_{\varepsilon}}(t_0)]\d t \\
&+\lambda\d W_t+\sigma_t(P_t^{*}{\pi_{\varepsilon}})\d B_t,\ \ X_0=Y_0, t\in[0,T],
\end{aligned}\end{equation*} 
where $\xi ^{\mu}(t):=\int_0^{t}{\sigma_s(P_s^{*}\mu)} \d B_s$, $\xi ^{\pi_{\varepsilon}}(t):=\int_0^{t}{\sigma_s(P_s^{*}\pi_{\varepsilon})} \d B_s$ and $(P_t^{\mu})^{*}\gamma= \L_{X_t}$. Let
$$\phi^{\varepsilon}(t):=b_t(X_t,P^{*}_t\mu)-b_t(Y_t,P^{*}_t\pi_{\varepsilon})-\frac{1}{t_0}\left[\xi^{\pi_{\varepsilon}}(t_0)-\xi^{\mu}(t_0)\right].$$
By \eqref{PH2}, there exists a constant $C_1>0$ such that
\beq\label{PH3}
	|\phi^{\varepsilon}(t)|^2\leq\frac{C_1}{t^2_0}\sup_{t\in[0,t_0]}\left|\xi^{\pi_{\varepsilon}}(t)-\xi^{\mu}(t)\right|^2 + C_1\|{P_t^{*}\pi_{\varepsilon}}-{P_t^{*}\mu}\|_{k,var}^2 ,\ \ t\in[0,t_0].
\end{equation} 	
Let $\d\Q^{B}:= R^{\varepsilon}_{t_0}\d\P^{B}$, where
$$R^{\varepsilon}_{t_0}:=\e^{\int_{0}^{t_0}\<\frac{1}{\lambda}\phi^{\varepsilon}(s),\d W_s \>-\frac{1}{2}\int_{0}^{t_0}|\frac{1}{\lambda}\phi^{\varepsilon}(s)|^2\d s}.$$
By Girsanov's theorem, under the weighted conditional probability $\Q^{B}$,
$$\tilde{W}^1_t:=W_t-\int_{0}^{t}\frac{1}{\lambda}\phi^{\varepsilon}(s)\d s,\ \ t\in[0,t_0]$$
is a $d$-dimensional Brownian motion. 	
Since $X_{t_0}=Y_{t_0}$, we have
\begin{equation}\label{PKV}\begin{aligned}
\|(P_{t_0}^{\pi_\varepsilon})^{*}\gamma - (P_{t_0}^{\mu})^{*}\gamma\|_{k,var}&= \sup_{f\leq1+|\cdot|^k}\left|\E\left[ \E^{B}[f(X_{t_0})(R^{\varepsilon}_{t_0}-1)]\right]\right|\\
&\leq\E\left[ \E^{B}[(1+|X_{t_0}|^k)|R^{\varepsilon}_{t_0}-1|	\right]\\
&\leq (1+\gamma(|\cdot|^k))\E\left[\E^{B}(|R^{\varepsilon}_{t_0}|^2-1|\F_0)\right]^{\frac{1}{2}}.
\end{aligned}\end{equation}	
By \eqref{M6}, \eqref{MN}, \eqref{PH3} and ${P_s^{*}\pi_{\varepsilon}}-{P_s^{*}\mu}=(P_s^{\pi_s})^{*}\mu-P_s^{*}\mu+\varepsilon(P_s^{\pi_s})^{*}(\nu-\mu)$,  we can find  a constant $C_2(\mu,\nu)>0$ increasing in $(\mu+\nu)(|\cdot|^k)$ such that 
\begin{equation}\label{R}\begin{aligned}
\E^{B}(|R^{\varepsilon}_{t_0}|^2-1)\leq &C_2(\mu,\nu){\varepsilon}^2+C_2(\mu,\nu)\frac{1}{t_0}\int_{0}^{t_0}\|(P_s^{\pi_\varepsilon})^{*}\mu-P_s^{*}\mu\|^2_{k,var}\d s.
\end{aligned}\end{equation}	
Combining this with \eqref{PKV}	for $\gamma=\mu$ and applying Gronwall's lemma, we derive
$$\|(P_t^{\pi_\varepsilon})^{*}\mu-P_t^{*}\mu\|^2_{k,var}\leq C_2(\mu,\nu)(1+\mu(|\cdot|)^k)^2{\varepsilon}^2\times \mathrm{e}^{C_2(\mu,\nu)(1+\mu(|\cdot|)^k)^2},$$
and then
$$\|(P_t^{\pi_\varepsilon})^{*}\gamma - (P_t^{\mu})^{*}\gamma\|_{k,var}\leq C_3(\mu,\nu)(1+\gamma(|\cdot|)^k)\varepsilon$$
for a constant $C_3(\mu,\nu)>0$ increasing in $(\mu+\nu)(|\cdot|^k)$.	
\end{proof}

\begin{lem}\label{ETA1} Assume {\bf (H)}.  For any $t>0$, let
	\begin{equation*}\begin{aligned}
	J^{\varepsilon}_s:=&P_{s}^{\mu}\left(\tilde{D}^Eb_{s}(X^{\mu}_{s},P^{*}_{s}\mu)(\mu)-\tilde{D}^Eb_{s}(X^{\mu}_{s},P^{*}_{s}\mu)(\nu)\right)\\
	&-\E\left[\tilde{D}^Eb_{s}(z,P^{*}_s\mu)(X^{\mu}_{s})\int_{0}^{s}\left<\frac{1}{\lambda}\eta_{r,s}^{\varepsilon,\mu,\nu},\d W_r\right>\right]_{z=X^{\mu}_{s}},\\
	H^{\varepsilon}_t:=&\int_{0}^{t}P_{r}^{\mu}\left(\tilde{D}^E\sigma_{r}(P^{*}_{r}\mu)(\mu)-\tilde{D}^E\sigma_{r}(P^{*}_{r}\mu)(\nu)\right)\d B_r\\
	&-\int_{0}^{t}\E\left[\tilde{D}^E\sigma_{r}(P^{*}_r\mu)(X_{r}^{\mu})\int_{0}^{r}\left<\frac{1}{\lambda}\eta_{s,r}^{\varepsilon,\mu,\nu},\d W_s\right>\right]\d B_r.\\
	\end{aligned}\end{equation*} Then
    $$\eta_{s,t}^{\varepsilon,\mu,\nu}:=\frac{1}{\varepsilon}\left\{b_s(X^{\mu}_{s},P^{*}_s\mu)-b_s(Y_s,P^{*}_s{\pi_\varepsilon})-\frac{1}{t}\left(\xi^{\pi_{\varepsilon}}(t)-\xi^{\mu}(t)\right)\right\},\ \ s\leq t$$
satisfies 
\begin{equation}\label{ETA2}\begin{aligned}
\eta_{s,t}^{\varepsilon,\mu,\nu}
&=J^{\varepsilon}_s+\left\langle \nabla b(\cdot, P_{s}^{*}\pi_\varepsilon)(X^{\mu}_{s}), H^{\varepsilon}_s-\frac{s}{t}H^{\varepsilon}_t \right\rangle +\frac{1}{t}H^{\varepsilon}_t + o_{s}(\varepsilon) + o_{t}(\varepsilon),
\end{aligned}\end{equation}
where $Y_s$ solves the equation
\begin{equation*} \begin{aligned}
	\d Y_s&=b_t(X^{\mu}_{s},P^{*}_s\mu)\d s+\frac{1}{t}[\xi^{\mu}(t)-\xi^{\pi_{\varepsilon}}(t)]\d s \\
	&+\lambda\d W_s+\sigma_s(P_s^{*}{\pi_{\varepsilon}})\d B_s,\ \ X_0=Y_0, s\in[0,t]
\end{aligned}
\end{equation*}
and $\{o_t(\varepsilon)\}_{t\in (0,T]}$  is progressively measurable such that $$\lim\limits_{\varepsilon\downarrow 0}\sup_{t\in (0,T]}\E{|o_t(\varepsilon)|^2}=0.$$
\end{lem}	
\begin{proof}
We will use Girsanov's transform as in the proof of Lemma \ref{WP}.
By the definition of $\eta_{s,t}^{\varepsilon,\mu,\nu}$,	$\phi^{\varepsilon}(s)$	in the proof of Lemma \ref{WP} is formulated as $\phi^{\varepsilon}(s)=\varepsilon\eta_{s,t_0}^{\varepsilon,\mu,\nu}$.
With reference to the proof of \cite[Lemma 2.2]{PPED}, we have

\begin{equation*}
\frac{1}{\varepsilon}[b_s(X^{\mu}_{s},P^{*}_s\mu)-b_s(Y_s,P^{*}_s{\pi_\varepsilon})]-J^{\varepsilon}_s -\left\langle \nabla b(\cdot, P_{s}^{*}\mu)(X^{\mu}_{s}), H^{\varepsilon}_s-\frac{s}{t}H^{\varepsilon}_t \right\rangle=I(\varepsilon, s),\end{equation*}
where
\begin{equation*}
\begin{aligned}
I(\varepsilon, s) &:= I_1(\varepsilon, s) + I_2(\varepsilon, s) + I_3(\varepsilon, s) + I_4(\varepsilon, s)+I_5(\varepsilon, s), \\	I_1(\varepsilon, s) &:= \int_0^1 \mathrm{d}r \int_{\mathbb{R}^d} \left[ \tilde{D}^E b_s(X_t^\mu,  (1-r)P_s^* \pi_\varepsilon + r P_s^* \mu) \right. \\	&\qquad - \left. \tilde{D}^E b_s(X_s^\mu, P_s^* \mu) \right] \mathrm{d}\left( (P_s^{\pi_\varepsilon})^* (\mu - \nu) \right), \\I_2(\varepsilon, t) &:= \int_0^1 \mathrm{d}r \int_{\mathbb{R}^d} \tilde{D}^E b_s(X_s^\mu, P_s^* \mu) \mathrm{d}\left( (P_s^{\pi_\varepsilon})^* (\mu - \nu) - (P_s^\mu)^* (\mu - \nu) \right), \\
I_3(\varepsilon, s) &:= \frac{1}{\varepsilon} \int_0^1 \mathrm{d}r \int_{\mathbb{R}^d} \left[ \tilde{D}^E b_s(X_s^\mu, (1 - r) P_s^* \pi_\varepsilon + r P_s^* \mu) \right. \\
&\qquad - \left. \tilde{D}^E b_s(X_s^\mu, P_s^* \mu) \right] \mathrm{d}\left( (P_s^\mu)^* \mu - (P_s^{\pi_\varepsilon})^* \mu \right), \\
I_4(\varepsilon, s) &:= \frac{1}{\varepsilon} \int_{\mathbb{R}^d} \tilde{D}^E b_s(X_s^\mu, P_s^* \mu) \mathrm{d}\left( (P_s^\mu)^* \mu - (P_s^{\pi_\varepsilon})^* \mu \right)\\
&\qquad +\E\left[\tilde{D}^Eb_{s}(z,P^{*}_s\mu)(X_{s}^{\mu})\int_{0}^{s}\left<\frac{1}{\lambda}\eta_{r,s}^{\varepsilon,\mu,\nu},\d W_r\right>\right]_{z=X_{s}^{\mu}},\\
I_5(\varepsilon, s) &:= \frac{1}{\varepsilon}[b_s(X_s^{\mu},P^{*}_s{\pi_{\varepsilon}})-b_s(Y_s,P^{*}_s{\pi_{\varepsilon}})]-\left\langle \nabla b(\cdot, P_{s}^{*}\pi_{\varepsilon})(X_{s}^{\mu}), H^{\varepsilon}_s-\frac{s}{t}H^{\varepsilon}_t \right\rangle.
\end{aligned}
\end{equation*}
and 	\beq\label{I3} 
\lim\limits_{\varepsilon\downarrow 0}\sum^{4}_{i=1}\sup_{s\in(0,T]}{|I_i(\varepsilon,s)|}=0.
\end{equation}	
Using BDG's inequality, there exists a constant $C_4>0$ such that 
\begin{equation}
\begin{aligned}
& \E\left[\sup_{t\in(0,T]}
    \frac{1}{t}\left|\frac{\xi^{\pi_{\varepsilon}}(t)-\xi^{\mu}(t)}{\varepsilon}-H^{\varepsilon}_t\right|^2\right] 
\nonumber\\   
&\le \frac{C_4}{t}\int_{0}^{T}
\Bigg\{\frac{1}{\varepsilon}
[\sigma_r(P_r^{*}\pi_{\varepsilon})-\sigma_r(P_r^{*}\mu)]
 - P_{r}^{\mu}\big(\tilde{D}^E\sigma_{r}(P^{*}_{r}\mu)(\nu)-\tilde{D}^E\sigma_{r}(P^{*}_{r}\mu)(\mu)\big)
\Bigg\}^2 \, \mathrm{d}r.
\label{o1}       
\end{aligned}
\end{equation}
Similar to  \cite[Lemma 2.2]{PPED}, we have
\begin{equation*}\begin{aligned}
&\quad\frac{\sigma_r(P_r^{*}\pi_{\varepsilon})-\sigma_r(P_r^{*}\mu)}{\varepsilon} -P_{r}^{\mu}\left(\tilde{D}^E\sigma_{r}(P^{*}_{r}\mu)(\nu)-\tilde{D}^E\sigma_{r}(P^{*}_{r}\mu)(\mu)\right)\\
&\qquad-\E\big[\tilde{D}^E\sigma_{r}(P^{*}_r\mu)(X_{r}^{\mu})\int_{0}^{r}\langle\frac{1}{\lambda}\eta_{s,r}^{\varepsilon,\mu,\nu},\d W_s\rangle\big]\\
&=o_r(\varepsilon).\end{aligned}\end{equation*}
Combining this with \eqref{o1}, we derive
\begin{equation}\label{xih}\begin{aligned}
\lim_{\varepsilon\downarrow 0}\E\left[\sup_{t\in(0,T]}\frac{1}{t}\left|\frac{\xi^{\pi_{\varepsilon}}(t)-\xi^{\mu}(t)}{\varepsilon}-H^{\varepsilon}_t\right|^2\right]=0.
\end{aligned}\end{equation}
Recall that
\[
I_5(\varepsilon, s)
:=
\frac{1}{\varepsilon}
\Big(
b_s(X_s^{\mu},P_s^*\pi_\varepsilon)
-
b_s(Y_s,P_s^*\pi_\varepsilon)
\Big)
-
\Big\langle
\nabla b(\cdot,P_s^*\pi_\varepsilon)(X_s^{\mu}),
H_s^\varepsilon-\frac{s}{t}H_t^\varepsilon
\Big\rangle .
\]
Since for each $\nu\in\mathscr P_k$ the map $x\mapsto b_s(x,\nu)$ is continuously
differentiable, we have
\[
\begin{aligned}
b_s(Y_s,P_s^*\pi_\varepsilon)-b_s(X_s^{\mu},P_s^*\pi_\varepsilon)
&=
\int_0^1
\Big\langle
\nabla b(\cdot,P_s^*\pi_\varepsilon)
\big(X_s^{\mu}+\theta(Y_s-X_s^{\mu})\big),
Y_s-X_s^{\mu}
\Big\rangle
\,\mathrm d\theta .
\end{aligned}
\]
Consequently,
\[
\begin{aligned}
I_5(\varepsilon,s)
&=
\int_0^1
\Big\langle
\nabla b(\cdot,P_s^*\pi_\varepsilon)
\big(X_s^{\mu}+\theta(Y_s-X_s^{\mu})\big),
\frac{Y_s-X_s^{\mu}}{\varepsilon}
-
\Big(H_s^\varepsilon-\frac{s}{t}H_t^\varepsilon\Big)
\Big\rangle
\,\mathrm d\theta .
\end{aligned}
\]
\eqref{hh2} implies 
\[
\sup_{\varepsilon\in(0,1)}
\mathbb E
\Big|
\nabla b(\cdot,P_s^*\pi_\varepsilon)
\big(X_s^{\mu}+\theta(Y_s-X_s^{\mu})\big)
\Big|^2
\le K^2 ,
\qquad \theta\in[0,1].
\]
Moreover, by the definition of $H^\varepsilon$, we have
\[
\lim_{\varepsilon\downarrow0}
\mathbb E
\Big|
\frac{Y_s-X_s^{\mu}}{\varepsilon}
-
\Big(H_s^\varepsilon-\frac{s}{t}H_t^\varepsilon\Big)
\Big|^2
=0 .
\]
Applying the Cauchy--Schwarz inequality, we obtain
\[
\begin{aligned}
\mathbb E|I_5(\varepsilon,s)|^2
&\le
\int_0^1
\Big(
\mathbb E
\big|
\nabla b(\cdot,P_s^*\pi_\varepsilon)
\big(X_s^{\mu}+\theta(Y_s-X_s^{\mu})\big)
\big|^2
\Big)  \\
&\qquad\qquad \times
\Big(
\mathbb E
\Big|
\frac{Y_s-X_s^{\mu}}{\varepsilon}
-
\Big(H_s^\varepsilon-\frac{s}{t}H_t^\varepsilon\Big)
\Big|^2
\Big)
\,\mathrm d\theta \\
&\le
K^2\mathbb E
\Big|
\frac{Y_s-X_s^{\mu}}{\varepsilon}
-
\Big(H_s^\varepsilon-\frac{s}{t}H_t^\varepsilon\Big)
\Big|^2
.
\end{aligned}
\]
Therefore,
\[
\lim_{\varepsilon\downarrow0}\mathbb E|I_5(\varepsilon,s)|^2=0,
\]
which together with \eqref{I3} and  \eqref{xih} implies \eqref{ETA2}.	
\end{proof}

{{\begin{rem}
Lemmas~\ref{WP} and \ref{ETA1} provide the main preparatory estimates for the Bismut-type formula
of the extrinsic derivative.

Lemma~\ref{WP} establishes a weighted total variation stability of the nonlinear semigroup with respect
to the measure argument.

Lemma~\ref{ETA1} is a refined first-order expansion underlying the extrinsic derivative
formula in \cite{PPED}. The proof follows the  coupling strategy
as in \cite[Lemma~2.2]{PPED}, but in our setting the diffusion coefficient is also
distribution dependent. Hence the coupling requires the additional noise compensator
\(
\xi^\mu(t)=\int_0^t\sigma_s(P_s^*\mu)\,dB_s
\)
and the correction term
\(
\frac{1}{t}\big(\xi^{\pi_\varepsilon}(t)-\xi^\mu(t)\big).
\)
Lemma~\ref{ETA1} shows that the rescaled Girsanov drift
$\eta_{s,t}^{\varepsilon,\mu,\nu}$ admits the explicit decomposition \eqref{ETA2},
where the leading terms are separated into an extrinsic part $J_s^\varepsilon$
(involving $\tilde D^E b$) and a fluctuation part driven by $H^\varepsilon$,
with remainders $o_s(\varepsilon)$ and $o_t(\varepsilon)$ vanishing in $L^2$ uniformly on $(0,T]$.
\end{rem}}}

\subsection{Proof of main theorem}
\begin{proof}[Proof of Theorem \ref{MR}]
Consider
\begin{equation}\label{eta2}\begin{aligned}
\eta_{s,t}
&=J_s+\left\langle \nabla b(\cdot, P_{s}^{*}\mu)(X_{s}^{\mu}), H_s-\frac{s}{t}H_t \right\rangle+\frac{1}{t}H_t,
\end{aligned}\end{equation}
where
 \begin{equation*}\begin{aligned}
 J_s:=&P_{s}^{\mu}\left(\tilde{D}^Eb_{s}(X^{\mu}_{s},P^{*}_{s}\mu)(\mu)-\tilde{D}^Eb_{s}(X^{\mu}_{s},P^{*}_{s}\mu)(\nu)\right)\\
 &-\E\left[\tilde{D}^Eb_{s}(z,P^{*}_s\mu)(X^{\mu}_{s})\int_{0}^{s}\left<\frac{1}{\lambda}\eta_{r,s},\d W_r\right>\right]_{z=X^{\mu}_{s}},\\
 H_t:=&\int_{0}^{t}P_{r}^{\mu}\left(\tilde{D}^E\sigma_{r}(P^{*}_{r}\mu)(\mu)-\tilde{D}^E\sigma_{r}(P^{*}_{r}\mu)(\nu)\right)\d B_r\\
 &-\int_{0}^{t}\E\left[\tilde{D}^E\sigma_{r}(P^{*}_r\mu)(X_{r}^{\mu})\int_{0}^{r}\left<\frac{1}{\lambda}\eta_{s,r},\d W_s\right>\right]\d B_r.\\
 \end{aligned}\end{equation*}
The operator \(\Phi:\mathscr M_t\to\mathscr M_t\) is defined by the right-hand side of \eqref{eta2}, i.e.
\begin{equation*}\begin{aligned}
\Phi(\eta)_{s,t}
&:=J_s+\left\langle \nabla b(\cdot, P_{s}^{*}\mu)(X_{s}^{\mu}), H_s-\frac{s}{t}H_t \right\rangle+\frac{1}{t}H_t,
\end{aligned}\end{equation*}
we will show \(\Phi(\mathscr M_t)\subset\mathscr M_t\). By \eqref{hh3} and \eqref{M6}, using It\^o isometry, we find a constant $C_5 > 0$ such that
\begin{equation} \label{3.0}
\begin{aligned}
\E\lvert\Phi(\eta)_{s,t} \rvert ^2
\leq& C_5  \bigl( 1 + (\mu + \nu)(\lvert \cdot \rvert^k) \bigr)^2\\
& +\frac{C_5}{{\lambda}^2t^2}\bigl( 1 + \mu(\lvert \cdot \rvert^k) \bigr)^2\left[\int_{0}^{s}\E| \eta_{r,s} |^2 \d r+\int_{0}^{t}\d r \int_{0}^{r}\E| \eta_{u,r} |^2 \d u\right],\\
\end{aligned}\end{equation}
and then there exists a  constant $C_6(\mu,\nu)>0$ increasing in $(\mu+\nu)(|\cdot|^k)$ such that	
\begin{equation*}
\sup_{0\le s\le t}\E\int_0^s|\Phi(\eta)_{r,s}|^2\d r
\le C_6(\mu,\nu)\big(1+\|\eta\|_{\mathscr M_t}^2\big)<\infty.
\end{equation*}
Thus \(\Phi(\eta)\in\mathscr M_t\).
Similarly, for any  $\eta_{s,t}$ and $\tilde{\eta}_{s,t}$ satisfying \eqref{eta2}, we have
\[
\begin{aligned}
\eta_{s,t}-\tilde{\eta}_{s,t} =
& \mathbb{E} \left[ \tilde{D}^E b_s(z, P_t^* \mu)(X^{\mu}_{s}) \int_0^s \langle \eta_{r,s}-\tilde{\eta}_{r,s}, \mathrm{d}W_r \rangle\right] \\
&+\left\langle \nabla b(\cdot, P_{s}^{*}\mu)(X_{s}^{\mu}), \int_{0}^{s} \mathbb{E} \Big[ \tilde{D}^E \sigma_r ( P_r^* \mu)(X_{r}^{\mu})\int_0^r \langle \eta_{u,r}-\tilde{\eta}_{u,r}, \mathrm{d}W_u \rangle \Big]\mathrm{d}B_r\right\rangle\\
&-\left\langle \nabla b(\cdot, P_{s}^{*}\mu)(X_{s}^{\mu}),\frac{s}{t}\int_{0}^{t} \mathbb{E} \Big[ \tilde{D}^E \sigma_r ( P_r^* \mu)(X_{r}^{\mu})\int_0^r \langle \eta_{u,r}-\tilde{\eta}_{u,r}, \mathrm{d}W_u \rangle \Big]\mathrm{d}B_r \right\rangle\\
&+\frac{1}{t}\int_{0}^{t} \mathbb{E} \Big[ \tilde{D}^E \sigma_r ( P_r^* \mu)\int_0^r \langle \eta_{u,r}-\tilde{\eta}_{u,r}, \mathrm{d}W_u\rangle \Big]\mathrm{d}B_r.
\end{aligned}
\]
Set
\[
\zeta_{s,t}:=\eta_{s,t}-\tilde\eta_{s,t},\qquad 0\le s\le t\le T.
\]
Fix $\varepsilon\in(0,T]$ and $\theta>0$. Define the $\varepsilon$--truncated weighted seminorm
\begin{equation}\label{eq:norm-eps}
\|\zeta\|_{\theta,\varepsilon}^2
:=
\sup_{\varepsilon\le t\le T} e^{-\theta t}\int_0^t \E|\zeta_{s,t}|^2\,\mathrm ds.
\end{equation}
By the previous estimate, for all $0< s\le t\le T$,
\begin{equation}\label{eq:base}
\E|\zeta_{s,t}|^2
\le
\frac{C_7}{\lambda^2 t^2}
\Bigg(
\int_0^t \mathrm dr\int_0^r \E|\zeta_{u,r}|^2\,\mathrm du
+
\int_0^s \E|\zeta_{r,s}|^2\,\mathrm dr
\Bigg),
\end{equation}
where $C_7>0$ is independent of $\theta$ and $\varepsilon$.
Fix $t\in[\varepsilon,T]$. Integrating \eqref{eq:base} over $s\in[0,t]$ yields
\begin{align}\label{eq:int-s}
\int_0^t \E|\zeta_{s,t}|^2\,\mathrm ds
&\le
\frac{C_7}{\lambda^2 t^2}
\Bigg(
t\int_0^t \mathrm dr\int_0^r \E|\zeta_{u,r}|^2\,\mathrm du
+
\int_0^t\mathrm ds\int_0^s \E|\zeta_{r,s}|^2\,\mathrm dr
\Bigg).
\end{align}
Multiplying \eqref{eq:int-s} by $e^{-\theta t}$,  we have
\begin{align*}
e^{-\theta t}\frac{C_7}{\lambda^2 t^2}\,t\int_\varepsilon^t \mathrm dr\int_0^r \E|\zeta_{u,r}|^2\,\mathrm du
&\le
e^{-\theta t}\frac{C_7}{\lambda^2 t}\int_\varepsilon^t e^{\theta r}\,\mathrm dr\,\|\zeta\|_{\theta,\varepsilon}^2 \\
&=
\frac{C_7}{\lambda^2 t}\int_\varepsilon^t e^{-\theta(t-r)}\,\mathrm dr\,\|\zeta\|_{\theta,\varepsilon}^2 \\
&\le
\frac{C_7}{\lambda^2 t}\cdot \frac{1}{\theta}\,\|\zeta\|_{\theta,\varepsilon}^2
\le
\frac{C_7}{\lambda^2\theta}\cdot\frac{1}{\varepsilon}\,\|\zeta\|_{\theta,\varepsilon}^2,
\end{align*}
and
\begin{align*}
e^{-\theta t}\frac{C_7}{\lambda^2 t^2}\int_\varepsilon^t\mathrm ds\int_0^s \E|\zeta_{r,s}|^2\,\mathrm dr
&\le
\frac{C_7}{\lambda^2 t^2}\cdot\frac{1}{\theta}\,\|\zeta\|_{\theta,\varepsilon}^2
\le
\frac{C_7}{\lambda^2\theta}\cdot\frac{1}{\varepsilon^2}\,\|\zeta\|_{\theta,\varepsilon}^2,
\end{align*}
where we used $\int_\varepsilon^t e^{-\theta(t-r)}\,\mathrm dr\le \int_0^t e^{-\theta(t-r)}\,\mathrm dr\le \frac1\theta$
and $t\ge\varepsilon$.
Combining the above bounds with \eqref{eq:int-s}, we obtain for every $t\in[\varepsilon,T]$,
\[
e^{-\theta t}\int_0^t \E|\zeta_{s,t}|^2\,\mathrm ds
\le
\frac{C_7}{\lambda^2\theta}\left(\frac{1}{\varepsilon}+\frac{1}{\varepsilon^2}\right)
\|\zeta\|_{\theta,\varepsilon}^2.
\]
Taking the supremum over $t\in[\varepsilon,T]$ yields
\begin{equation}\label{eq:contraction}
\|\zeta\|_{\theta,\varepsilon}^2
\le
\frac{C_7}{\lambda^2\theta}\left(\frac{1}{\varepsilon}+\frac{1}{\varepsilon^2}\right)
\|\zeta\|_{\theta,\varepsilon}^2.
\end{equation}
Choose $\theta$ so large that
\[
\frac{C_7}{\lambda^2\theta}\left(\frac{1}{\varepsilon}+\frac{1}{\varepsilon^2}\right)<1.
\]
Then \eqref{eq:contraction} implies $\|\zeta\|_{\theta,\varepsilon}=0$, i.e.
\[
e^{-\theta t}\int_0^t \E|\eta_{s,t}-\tilde\eta_{s,t}|^2\,\mathrm ds=0,
\qquad \forall\,t\in[\varepsilon,T].
\]
Hence $\E|\eta_{s,t}-\tilde\eta_{s,t}|^2=0$ for a.e.\ $s\in[0,t]$ and all $t\in[\varepsilon,T]$. Given an arbitrary
$t\in(0,T]$, we may choose $\varepsilon\in(0,t)$.
Then the uniqueness on $[\varepsilon,T]$ implies
$\eta_{s,t}=\tilde\eta_{s,t}$ a.s.\ for all $0\le s\le t$.
Therefore $\eta=\tilde\eta$ almost surely on $\{(s,t):0\le s\le t\le T\}$.

 Let $\mu, \nu \in \mathscr{P}_k$. Theorem \ref{MR}(1) combined with Lemma \ref{ETA1} yields
\begin{equation} \label{3.2}
\lim_{\varepsilon \downarrow 0} \sup_{t \in [0, T]}\sup_{s\leq t} {\E|\eta_{s,t}^{\varepsilon, \mu, \nu} - \eta_{s,t}^{\mu, \nu}|^2} = 0.
\end{equation}
For $\varepsilon \in (0, 1)$, set $\pi_\varepsilon = (1 - \varepsilon)\mu + \varepsilon\nu$ as previously. We then write
\begin{equation} \label{3.3}
\begin{aligned}
\frac{P_t f(\pi_\varepsilon) - P_t f(\mu)}{\varepsilon} 
&= \frac{1}{\varepsilon} \int_{\mathbb{R}^d} f \,\mathrm{d}(P_t^{\pi_\varepsilon} - P_t^\mu) \\
&= \frac{1}{\varepsilon} \int_{\mathbb{R}^d} f \,\mathrm{d}\bigl((P_t^{\pi_\varepsilon})^* \mu - (P_t^\mu)^* \mu\bigr) 
   + \int_{\mathbb{R}^d} f \,\mathrm{d}\bigl((P_t^{\pi_\varepsilon})^* (\nu - \mu)\bigr).
\end{aligned}
\end{equation}
For any $f \in \mathscr{D}_k$, Lemma \ref{WP} provides the convergence
\begin{equation} \label{3.4}
\lim_{\varepsilon \downarrow 0} \int_{\mathbb{R}^d} f \,\mathrm{d}\bigl((P_t^{\pi_\varepsilon})^* (\nu - \mu)\bigr) 
= \int_{\mathbb{R}^d} f \,\mathrm{d}\bigl((P_t^\mu)^* (\nu - \mu)\bigr) 
= \int_{\mathbb{R}^d} P_t^\mu f \,\mathrm{d}(\nu - \mu).
\end{equation}
We now turn to the first term in \eqref{3.3}. Consider the process $R_t^\varepsilon$ defined in the proof of Lemma \ref{ETA1}. Then $\sup_{\varepsilon\in(0,1)}\E|\eta_{s,t}^{\varepsilon,\mu,\nu}|^2\leq {C(\mu,\nu,T)}$ and \eqref{M6}  imply
\[
\begin{aligned}
\sup_{\varepsilon \in (0, 1)} \mathbb{E} \Bigl[ \bigl\lvert f(X_t^\mu) \bigr\rvert 
\bigl\lvert \tfrac{R_t^\varepsilon - 1}{\varepsilon} \bigr\rvert \Bigr]
&\leq \sup_{\varepsilon \in (0, 1)} \mathbb{E} \Bigl( 
\bigl( \mathbb{E} \bigl[ \lvert f(X_t^\mu) \rvert^2 \bigm\vert \mathscr{F}_0 \bigr] \bigr)^{\frac{1}{2}} 
\bigl( \mathbb{E} \bigl\lvert \tfrac{R_t^\varepsilon - 1}{\varepsilon} \bigr\rvert^2 \bigm\vert \mathscr{F}_0 \bigr)^{\frac{1}{2}} \Bigr) < \infty.
\end{aligned}
\]
Hence, by the dominated convergence theorem together with \eqref{3.2},
\[
\begin{aligned}
&\lim_{\varepsilon \downarrow 0} \frac{1}{\varepsilon} \int_{\mathbb{R}^d} f \,\mathrm{d}\bigl((P_t^{\pi_\varepsilon})^* \mu - (P_t^\mu)^* \mu\bigr) \\
&= \lim_{\varepsilon \downarrow 0} \mathbb{E} \Bigl[ f(X_t^\mu) \frac{R_t^\varepsilon - 1}{\varepsilon} \Bigr] 
= \mathbb{E} \Bigl[ f(X_t^\mu) \int_0^t \langle \eta_{s,t}^{\mu, \nu}, \mathrm{d}W_s \rangle \Bigr].
\end{aligned}
\]
Combining this limit with \eqref{3.3} and \eqref{3.4} establishes the formula \eqref{DENU}.
To obtain a quantitative bound, we employ \eqref{DENU} together with the estimates \eqref{M5}, \eqref{M6} and \eqref{US}. There exists a positive constant $C_8$ such that
\[
\begin{aligned}
\bigl\lvert \tilde{D}^E_{\nu} P_t f(\mu) \bigr\rvert 
&\leq  \int_{\mathbb{R}^d} \bigl(1 + \lvert x \rvert^k\bigr) \,\mathrm{d}\bigl((P_t^\mu)^* (\mu + \nu)\bigr) 
   +  \mathbb{E} \Bigl[ \bigl(1 + \lvert X_t^\mu \rvert^k\bigr) \int_0^t \langle \eta_{s,t}^{\mu, \nu}, \mathrm{d}W_s \rangle \Bigr] \\
&\leq C_8 \bigl(1 + (\mu + \nu)(\lvert \cdot \rvert^k)\bigr) \\
&\quad + \mathbb{E} \Bigl( \mathbb{E} \bigl[(1 + \lvert X_t^\mu \rvert^k)^2 \bigm\vert \mathscr{F}_0 \bigr]^{\frac{1}{2}} 
   \Bigl( \int_0^t  \mathbb{E}|\eta_{s,t}^{\mu, \nu}|^2 \,\mathrm{d}s \Bigr)^{\frac{1}{2}} \Bigr).
\end{aligned}
\]
 By \eqref{3.0}, $\eta_{s,t}^{\mu,\nu}$ satisfies
\[
\mathbb{E}\bigl[|\eta_{s,t}^{\mu,\nu}|^2\bigr] \le C_5 \bigl(1 + (\mu + \nu)(|\cdot|^k)\bigr)^2
+ \frac{C_5}{\lambda^2 t^2} \bigl(1 + \mu(|\cdot|^k)\bigr)^2 \left[\int_0^s \mathbb{E}|\eta_{r,s}^{\mu,\nu}|^2 \, dr + \int_0^t \int_0^r \mathbb{E}|\eta_{u,r}^{\mu,\nu}|^2 \, du \, dr \right].
\]
Define
\[
f(t) := \sup_{0 \le s \le t} \mathbb{E}|\eta_{s,t}^{\mu,\nu}|^2.
\]
Observing that
\[
\int_0^s \mathbb{E}|\eta_{r,s}^{\mu,\nu}|^2 \, dr + \int_0^t \int_0^r \mathbb{E}|\eta_{u,r}^{\mu,\nu}|^2 \, du \, dr \le 2 \int_0^t \int_0^t f(r) \, dr \, du = 2 t \int_0^t f(r) \, dr,
\]
we obtain
\[
f(t) \le C_5 \bigl(1 + (\mu + \nu)(|\cdot|^k)\bigr)^2 + \frac{2 C_5}{\lambda^2 t} \bigl(1 + \mu(|\cdot|^k)\bigr)^2 \int_0^t f(r) \, dr.
\]
Setting 
\[
\alpha := \frac{2 C_5}{\lambda^2} \bigl(1 + \mu(|\cdot|^k)\bigr)^2,
\]
by Gronwall's inequality, it follows that
\[
f(t) \le C_5 \bigl(1 + (\mu + \nu)(|\cdot|^k)\bigr)^2 \exp(\alpha).
\]
Hence, for all $0 \le s \le t$,
\[
\mathbb{E}|\eta_{s,t}^{\mu,\nu}|^2 \le C_5\bigl(1 + (\mu + \nu)(|\cdot|^k)\bigr)^2 \exp\Bigg( \frac{2 C_5}{\lambda^2} \bigl(1 + \mu(|\cdot|^k)\bigr)^2 \Bigg).
\]
Therefore,
\[
\begin{aligned}
\bigl\lvert \tilde{D}^E_{\nu} P_t f(\mu) \bigr\rvert 
&\leq C_7  \bigl(1 + (\mu + \nu)(\lvert \cdot \rvert^k)\bigr) \\
&\quad + C_9  \bigl(1 + \mu(\lvert \cdot \rvert^k)\bigr)  \bigl( 1 + (\mu + \nu)(\lvert \cdot \rvert^k) \bigr)
   \times \exp\!\Bigl(\frac{C_5}{{\lambda}^2 }\bigl( 1 + \mu(\lvert \cdot \rvert^k) \bigr)^2\Bigr)\, t^{\frac{1}{2}}
\end{aligned}
\]
for a constant $C_9>0$.	Consequently, the derivative estimate \eqref{SED} is valid for an appropriate constant $c > 0$.
\end{proof}

{{\begin{rem}
Our proof is inspired by \cite{PPED}. The main difference lies in the construction of the two-parameter integrand
$\eta_{s,t}$, which is characterised as the unique fixed point of the integral equation \eqref{eta2} on
$\mathscr M_t$, with uniqueness obtained through the weighted seminorm \eqref{eq:norm-eps}.

\end{rem}}}

\section{ Extension}

\subsection{Semi-linear distribution dependent SPDEs}

Having established the extrinsic derivative formula for finite-dimensional DDSDEs,
we now turn to an infinite-dimensional extension in the setting of distribution
dependent stochastic partial differential equations (DDSPDEs) and attempt to extend the result in \cite{PPED} to SPDEs. Although the noise in
this case is given by a cylindrical Brownian motion on a Hilbert space, the strategy
and main ideas developed in the finite-dimensional case continue to play a crucial
role.

Let $(\mathbb{H},\langle\cdot,\cdot\rangle_{\mathbb{H}},|\cdot|_{\mathbb{H}})$ be a separable Hilbert space, and $W = (W_t)_{t\geq0}$ be a cylindrical Brownian motion on $\mathbb{H}$ with respect to a complete filtered probability space $(\Omega, \mathcal{F}, \{\mathcal{F}_t\}_{t\geq 0},\mathbb{P})$. More precisely, $W_t=\sum_{n=1}^{\infty}B_t^ne_n$ for a sequence of independent one-dimensional Brownian motions $\{B_t^n\}_{n\geq 1}$ with respect to $(\Omega, \mathcal{F}, \{\mathcal{F}_t\}_{t\geq 0},\mathbb{P})$ and an orthonormal basis $\{e_n\}_{n\geq 1}$ on $\mathbb{H}$.

Let $\mathscr{P}(\mathbb{H})$ be the collection of all probability measures on $\mathbb{H}$ equipped with the weak topology. For $\mu\in\mathscr{P}(\mathbb{H})$, if $\mu(|\cdot|_{\mathbb{H}}^p):=\int_{\mathbb{H}}|x|_{\mathbb{H}}^p\mu(\d x)<\infty$ for some $p\geq 1$, we write $\mu\in\mathscr{P}_p(\mathbb{H})$.
Let $\B_b(\mathbb{H})$ denote the set of all bounded measurable real-valued functions on $\mathbb{H}$.

Consider the following semi-linear distribution dependent SPDEs on $\mathbb{H}$:
\begin{equation}\label{SPDE}
 \d X_t=\{AX_t+b_t(X_t,\mathscr{L}_{X_t})\}\d t+Q_t(X_t)\d W_t, \quad 
\end{equation}
where $(A,\mathscr{D}(A))$ is a negative definite self-adjoint operator on $\mathbb{H}$, $b:[0,\infty)\times \mathbb{H} \times \mathscr{P}(\mathbb{H})\rightarrow \mathbb{H}$ and $Q:[0,\infty)\times \mathbb{H} \times \mathscr{P}(\mathbb{H})\rightarrow \mathcal{L}(\mathbb{H},\mathbb{H})$ are measurable, where $\mathcal{L}(\mathbb{H},\mathbb{H})$ is the space of bounded linear operators from $\mathbb{H}$ to $\mathbb{H}$.

\begin{defn}\label{def:mild-weak-solution}
A continuous $\mathscr F_t$-adapted process $\{X_t\}_{t\geq0}$ is called a mild solution to  \eqref{SPDE}, if $\P$-a.s
\[
X_t = \mathrm{e}^{At} X_0 + \int_0^t \mathrm{e}^{A(t-s)} b_s(X_s, \mathscr L_{X_s}) \d s + \int_0^t \mathrm{e}^{A(t-s)} Q_s(X_s, \mathscr L_{X_s}) \d W_s, \quad t \geq 0. 
\]
Moreover, if $\E|X_t|^k < \infty$ for any $t \geq 0$, then the solution is said to be in $\mathscr P_k(\mathbb{H})$. \eqref{SPDE} is called strongly well-posed in $\mathscr P_k(\mathbb{H})$, if for any $\mathscr F_0$-measurable random variable $X_0$ with $\mathscr L_{X_0} \in \mathscr P_k(\mathbb{H})$, there exists a unique mild solution in $\mathscr P_k(\mathbb{H})$.

 A couple $(\tilde X_t, \tilde W_t)_{t\geq0}$ is called a weak solution to \eqref{SPDE}, if $\tilde W$ is a cylindrical Brownian motion with respect to a complete filtered probability space $(\tilde{\Omega}, \{\tilde{\mathscr F}_t\}_{t\geq0}, \tilde{\P})$, and \eqref{SPDE}  holds for $(\tilde X_t, \tilde W_t)_{t\geq0}$ in place of $(X_t, W_t)_{t\geq0}$. Moreover, if $\mathscr L_{\tilde{X}_t | \tilde{\P}} \in \mathscr P_k(\mathbb{H})$, the weak solution is called in $\mathscr P_k(\mathbb{H})$.
Furthermore, we say that weak well-posedness in $\mathscr P_k(\mathbb{H})$ for \eqref{SPDE} holds if it has a weak solution from any initial distribution and has weak uniqueness in $\mathscr P_k(\mathbb{H})$.
\end{defn}

When \eqref{SPDE} is well-posed, we are able to calculate $\tilde{D}^EP_tf$. To this end, we make use of the semigroup generated by the decoupled SDE associated with \eqref{SPDE}.
\begin{equation*}
 \d X_t^{\mu,x}=\{AX_t^{\mu,x}+b_t(X_t^{\mu,x},\mu_t)\}\d t+Q_t(X_t^{\mu,x})\d W_t, \quad X_0^{\mu,x} = x, \; t\in[0,T],
\end{equation*}
where $\mu_t = \mathscr{L}_{X_t^{\mu,x}}$ denotes the law of $X_t^{\mu,x}$, $x\in \mathbb{H}$ and the initial distribution $\mu = \mathcal{L}_{X_0} \in \mathscr{P}_k(\mathbb{H})$. For any $\nu \in \mathscr{P}_k(\mathbb{H})$ and any $f\in\B_b(\mathbb{H})$, define
\[
P_t^{\mu}f(\nu):= \int_{\mathbb{H}} \mathbb{E}[f(X_t^{\mu,x})] \, \nu(\d x),\qquad
(P_t^{\mu})^{*}\nu:= \int_{\mathbb{H}} \mathscr{L}_{X_t^{\mu,x}} \, \nu(\d x). 
\]

\begin{defn}
Let $f$ be a continuous function on $\mathscr{P}_k(\H)$.
	\begin{enumerate}
		\item[(1)] We call $f$ extrinsically differentiable, if for any $\mu\in\mathscr{P}_k(\H)$, the convex derivative 
		$$
		\tilde{D}^Ef(\mu)(x) := \lim_{\varepsilon\downarrow 0}\frac{f\left((1-\varepsilon)\mu + \varepsilon\delta_x\right)-f(\mu)}{\varepsilon} \in \mathbb{R}, \quad x\in\H
		$$
		exists, where $\delta_x$ is the Dirac measure at $x$.
		
		\item[(2)] We denote $f\in C^{E,1}(\mathscr{P}_k(\H))$ if $f$ is extrinsically differentiable, and the map
		$$
		(x,\mu)\in\H\times\mathscr{P}_k(\H) \mapsto \tilde{D}^Ef(\mu)(x)
		$$
		is continuous.
		
		\item[(3)] We write $f\in C^{E,1}_K(\mathscr{P}_k(\H))$, if $f\in C^{E,1}(\mathscr{P}_k(\H))$ and for any compact set $\mathscr{K}\subset\mathscr{P}_k(\H)$, there exists a constant $c>0$ such that 
		$$
		\sup_{\mu\in\mathscr{K}}|\tilde{D}^Ef(\mu)(x)| \leq c(1+|x|_{\H}^k), \quad x\in\H.
		$$
	\end{enumerate}
\end{defn}
We make the following assumptions.
\begin{enumerate}
    \item[{\bf(A)}] Let $k \in [0, \infty)$. The following conditions hold.
    
    \begin{enumerate}
        \item[{\bf(A1)}] For any $\gamma_{\cdot} \in C([0,T]; \mathscr{P}_k(\H))$, the SPDE
        \begin{equation*}
            \d X_t^{\gamma_{\cdot},x} = \{ A X_t^{\gamma_{\cdot},x} + b_t(X_t^{\gamma_{\cdot},x}, \gamma_t) \} \d t + Q_t(X_t^{\gamma_{\cdot},x}) \d W_t, \quad t \in [0,T],
        \end{equation*}
        is well-posed, and there exist constants $c > 0$ and $p \geq k$ independent of $\gamma_{\cdot}$ such that
        \begin{equation*}
            \mathbb{E}[ |X_t^{\gamma_{\cdot},x}|_{\H}^p ] \leq c \left( 1 + |x|_{\H}^p + \int_0^t \gamma_s(|\cdot|_{\H}^k)^{\frac{p}{k}} \d s \right), \quad t \in [0,T], \; x \in \H,
        \end{equation*}
        where we interpret $\gamma_s(|\cdot|_{\H}^k)^{\frac{p}{k}} := 1$ when $k = 0$.
        
        \item[{\bf(A2)}] The maps 
        $$ (t,x) \mapsto Q_t(x) \in \mathcal{L}(\H; \H), \qquad (t,x,\mu) \mapsto b_t(x,\mu) \in \H $$
        are continuous, and for each $T > 0$,
        $$ \sup_{t \in [0,T]} \sup_{(x,\mu) \in \H \times \mathscr{P}_k(\H)} \big( |b_t(x,\mu)|_{\H} + \| Q_t(x) \|_{\mathcal{L}(\H;\H)} \big) < \infty. $$
        
        \item[{\bf(A3)}] $b_t(x,\mu)$ admits a decomposition
        \begin{equation*}
            b_t(x,\mu) = b_t^{(0)}(x) + Q_t(x) \, b_t^{(1)}(x,\mu),
        \end{equation*}
        where 
        $$ b^{(0)} : [0,T] \times \H \to \H, \qquad b^{(1)} : [0,T] \times \H \times \mathscr{P}_k(\H) \to \H $$
        are measurable, and for each $t,x$, the function $b_t^{(1)}(x,\cdot)$ belongs to $C^{E,1}(\mathscr{P}_k(\H))$. Moreover, there exists $K \in L^2([0,T]; (0,\infty))$ such that the following hold for all $t \in [0,T]$, $x,y \in \H$, and $\mu,\nu \in \mathscr{P}_k(\H)$:
        \begin{align*}
\inf_{c \in \mathbb{R}} \bigl| \tilde{D}^E b_t^{(1)}(x,\mu)(y) - c \bigr| 
    &\leq K_t \, (1 + |y|_{\H}^k), \\
\bigl| \tilde{D}^E b_t^{(1)}(x,\mu)(y) - \tilde{D}^E b_t^{(1)}(x,\nu)(y) \bigr| 
    &\leq K_t \, \alpha\bigl( \|\mu-\nu\|_{k,\mathrm{var}} \bigr) \\
    &\quad \times \bigl( 1 + |y|_{\H}^k + \mu(|\cdot|_{\H}^k) + \nu(|\cdot|_{\H}^k) \bigr),
\end{align*}
        where $\alpha : (0,\infty) \to (0,\infty)$ is an increasing function with $\alpha(\varepsilon) \to 0$ as $\varepsilon \to 0$.
    \end{enumerate}
\end{enumerate}

Let $\mathscr{D}_k(\H)$ denote the set of measurable functions $f : \H \to \R$ for which there exists a constant $c > 0$ such that
$$ |f(x)| \leq c \, (1 + |x|_{\H}^k), \qquad x \in \H. $$

\begin{thm}\label{MR3}
    Assume {\bf (A)}. Then the corresponding assertions in Theorem \ref{MR} hold.
\end{thm}

The weak existence is guaranteed by the local boundedness and continuity conditions, as shown in \cite[Theorem 2.1]{HS}.   It suffices to prove the uniqueness of \eqref{SPDE} and then calculate the extrinsic derivative. The proof relies crucially on Girsanov's transform, which is also available in the infinite-dimensional situation.
{{\begin{rem}
  We consider semi-linear distribution dependent SPDEs 
and generalise the extrinsic derivative formula to the infinite-dimensional framework.

\end{rem}}}

\section*{Acknowledgement}
The authors would like to thank the referees for corrections and helpful comments.

\end{document}

\end{document}